\documentclass{daj}

\usepackage{amsmath,amsthm,amsfonts,amssymb,xcolor,enumitem,mathscinet}
\usepackage{graphicx}
\usepackage{mathtools}
\usepackage{etoolbox}

\definecolor{PTblue}{HTML}{0077BB}
\definecolor{PTteal}{HTML}{009988}
\definecolor{PTmagenta}{HTML}{EE3377}

\newtheorem{theorem}{Theorem}[section]
\newtheorem{lemma}[theorem]{Lemma}
\newtheorem{corollary}[theorem]{Corollary}

\theoremstyle{definition}
\newtheorem{definition}[theorem]{Definition}

\theoremstyle{remark}
\newtheorem{remark}[theorem]{Remark}
\newtheorem{conjecture}[theorem]{Conjecture}
\newtheorem{example}[theorem]{Example}

\newcommand{\RR}{\mathbb{R}}
\newcommand{\QQ}{\mathbb{Q}}
\newcommand{\ZZ}{\mathbb{Z}}
\newcommand{\MM}{\mathbb{M}}
\newcommand{\XX}{\mathbb{X}}
\newcommand{\YY}{\mathbb{Y}}

\newcommand{\HH}{\mathbb{H}}
\newcommand{\EE}{\mathbb{E}}
\newcommand{\PP}{\mathbb{P}}
\newcommand{\Haus}{\mathcal{H}}
\newcommand{\Pack}{\mathcal{P}}

\newcommand{\interior}{\mathop\mathrm{int}\nolimits}
\newcommand{\dist}{\mathop\mathrm{dist}\nolimits}
\newcommand{\diam}{\mathop\mathrm{diam}\nolimits}
\newcommand{\side}{\mathop\mathrm{side}\nolimits}

\newcommand{\spt}{\mathop\mathrm{spt}\nolimits}

\newcommand{\res}{\hbox{ {\vrule height .22cm}{\leaders\hrule\hskip.2cm} }} 
\newcommand{\gap}{\mathop\mathrm{gap}\nolimits}
\newcommand{\muse}{\mu_{\mathbf{s}}}
\newcommand{\leaves}{\mathsf{Leaves}}
\newcommand{\Child}{\mathsf{Child}}
\newcommand{\Outer}{\mathsf{Outer}}
\newcommand{\Inner}{\mathsf{Inner}}
\newcommand{\Top}{\mathop\mathsf{Top}}
\newcommand{\radius}{\mathop\mathrm{radius}\nolimits}
\newcommand{\dash}{{\!-\!}}
\newcommand{\essinf}{\mathop\mathrm{ess\,inf}}
\newcommand{\esssup}{\mathop\mathrm{ess\,sup}}

\numberwithin{figure}{section}
\numberwithin{equation}{section}

\dajAUTHORdetails{%
  title = {Square Packings and Rectifiable Doubling Measures}, 
  author = {Matthew Badger and Raanan Schul},
  plaintextauthor = {Matthew Badger, Raanan Schul},
  keywords = {doubling measures, rectifiable measures, Lipschitz images, square packings, Cantor sets, quasi-Bernoulli measures, Hausdorff dimension, packing dimension, Assouad dimension},
}   

\dajEDITORdetails{%
   year={2025},
   number={3},
   received={6 September 2023},   
   published={9 May 2025},  
   doi={10.19086/da.137798},       
}   

\begin{document}

\begin{frontmatter}[classification=text]

\title{Square Packings and Rectifiable Doubling Measures} 

\author[badger]{Matthew Badger\thanks{Supported by NSF DMS grant 2154047.}}
\author[schul]{Raanan Schul\thanks{Supported by NSF DMS grant 2154613.}}

\begin{abstract}
We prove that for all integers $2\leq m\leq d-1$, there exist doubling measures on $\mathbb{R}^d$ with full support that are $m$-rectifiable and purely $(m-1)$-unrectifiable in the sense of Federer (i.e.~without assuming $\mu\ll\mathcal{H}^m$). The corresponding result for 1-rectifiable measures is originally due to Garnett, Killip, and Schul (2010). Our construction of higher-dimensional Lipschitz images is informed by a simple observation about square packing in the plane: $N$ axis-parallel squares of side length $s$ pack inside of a square of side length $\lceil N^{1/2}\rceil s$. The approach is robust and when combined with standard metric geometry techniques allows for constructions in complete Ahlfors regular metric spaces. One consequence of the main theorem is that for each $m\in\{2,3,4\}$ and $s<m$, there exist doubling measures $\mu$ on the Heisenberg group $\mathbb{H}^1$ and Lipschitz maps $f:E\subset\mathbb{R}^m\rightarrow\mathbb{H}^1$ such that $\mu\ll\mathcal{H}^{s-\epsilon}$ for all $\epsilon>0$, $f(E)$ has Hausdorff dimension $s$, and $\mu(f(E))>0$. This is striking, because $\Haus^m(f(E))=0$ for every Lipschitz map $f:E\subset\mathbb{R}^m\rightarrow\mathbb{H}^1$ by a theorem of Ambrosio and Kirchheim (2000). Another application of the square packing construction is that every compact metric space $\XX$ of Assouad dimension strictly less than $m$ is a Lipschitz image of a compact set $E\subset[0,1]^m$. Of independent interest, we record the existence of doubling measures on complete Ahlfors regular metric spaces with prescribed lower and upper Hausdorff and packing dimensions.
\end{abstract}
\end{frontmatter}



\section{Introduction}\label{s:intro}

In geometric measure theory, a fundamental problem is to detect the interaction of measures in $\RR^d$ or metric space $\XX$ with various canonical families of lower-dimensional sets such as rectifiable curves or $C^1$ submanifolds; see the surveys \cite{ident} and \cite{rectifiability-survey} for a detailed introduction to this topic and \cite{AM-JFA,Bate-tangents,BHS, KL-higher-co-dim-UR} for some of the latest advances. Questions that we might ask include: When does every set in the family have measure zero? When does some set in the family have positive measure? Classically, this problem was exclusively investigated within the class of Radon\footnote{On any \emph{proper} metric space $\XX$, i.e.~a metric space in which closed balls are compact, a \emph{Radon measure} $\mu$ is a Borel measure that is finite on bounded sets; see e.g.~\cite[Chapter 7]{Folland}.} measures $\mu$ on $\RR^d$ satisfying \begin{equation}\label{upper-density} 0<\limsup_{r\downarrow 0} \frac{\mu(B(x,r))}{r^m} <\infty\quad\text{$\mu$-a.e.},\end{equation} where $1\leq m\leq d-1$ is the dimension of the model sets and $B(x,r)$ denotes the closed ball with center $x$ and radius $r$; some highlights include \cite{MR44,Fed47,Preiss}. In particular, any measure satisfying \eqref{upper-density} is strongly $m$-dimensional in the sense that $\mu(E)=0$ for every set $E$ with $\Haus^m(E)=0$, but there exists a set $F$ such that $\Haus^m\res F$ is $\sigma$-finite and $\mu(\RR^d\setminus F)=0$; see e.g.~\cite[\S2.2]{BLZ}. For the definition of the $s$-dimensional Hausdorff measure $\Haus^s$ and its basic properties, see e.g.~\cite{Rogers} or \cite{Mattila}. Over the last decade, there has emerged an effort to study the concept of rectifiability within the larger class of arbitrary locally finite measures, without imposing the condition \eqref{upper-density}. We now possess complete pictures of the interaction of Radon measures in $\RR^d$ with rectifiable curves \cite{BS3} (also see \cite{Lerman,GKS,BS1,AM-rectifiable,BS2,MO-curves,BLZ}) and with $m$-dimensional Lipschitz graphs \cite{Badger-Naples,Dabrowski-cones} for arbitrary $1\leq m\leq d-1$. In both cases, the general solution was obtained after first studying the problem for doubling measures.

In this paper, we report some initial progress on the problem of testing when a Radon measure assigns full measure to a countable family of $m$-dimensional Lipschitz images. We leverage a new construction of Lipschitz maps from subsets of $\RR^m$, $m\geq 2$ (see \S\ref{s:Lipschitz}). Consistent with the convention used by Federer \cite[\S3.2.14]{Federer}, we say that a Borel measure $\mu$ on a metric space $\XX$ is \emph{$m$-rectifiable} if there exist bounded sets $E_i\subset\RR^m$ and Lipschitz maps $f_i:E_i\rightarrow\XX$ such that $\mu(\XX\setminus\bigcup_1^\infty f_i(E_i))=0$. We say that $\mu$ is \emph{purely $m$-unrectifiable} if $\mu(f(E))=0$ for every bounded set $E\subset\RR^m$ and Lipschitz map $f:E\rightarrow\XX$. By standard Lipschitz extension theorems, when $\XX=\RR^d$, one may replace the domain of the maps with $[0,1]^m$ or $\RR^m$ or with arbitrary sets $E\subset\RR^m$ without changing the class of rectifiable measures. Note that every $m$-rectifiable measure is automatically $(m+1)$-rectifiable by considering maps of the form $F(x_1,\dots,x_m,x_{m+1})=f(x_1,\dots,x_m)$. A measure $\mu$ on $\XX$ is \emph{doubling} if \eqref{doubling} below holds for some constant $1\leq D<\infty$. We say that $\XX$ is \emph{Ahlfors $q$-regular} if $q\in[0,\infty)$ and there exists a constant $C>1$ such that $C^{-1}r^q\leq \Haus^q(B(x,r))\leq C r^q$ for all $x\in\XX$ and $0<r<\diam X$. Our goal is to prove:

\begin{theorem}\label{t:main} Let $\XX$ be a complete Ahlfors $q$-regular metric space. For all integers $m\geq 1$ with $q>m-1$, there exists a doubling measure $\mu$ on $\XX$ such that $\mu$ is $m$-rectifiable and purely $(m-1)$-unrectifiable.

More precisely, for all integers $m\geq 1$ with $q>m-1$ and for all real-valued dimensions $0<s_H\leq s_P<q$ with $m-1<s_P< m$, we can find a Radon measure $\mu$ on $\XX$ and a constant $1\leq D<\infty$ depending only on $\XX$, $s_H$, and $s_P$ such that
\begin{equation}\label{doubling} 0<\mu(B(x,2r))\leq D\mu(B(x,r))<\infty\quad\text{for all $x\in\XX$ and $r>0$};\end{equation}
\begin{equation}\label{haus-dim-s}\liminf_{r\downarrow 0} \frac{\log\mu(B(x,r))}{\log r}=s_H\quad\text{at $\mu$-a.e.~$x\in\XX$;}\end{equation}
\begin{equation}\label{pack-dim-s}\limsup_{r\downarrow 0} \frac{\log\mu(B(x,r))}{\log r}=s_P\quad\text{at $\mu$-a.e.~$x\in\XX$;}\end{equation}
\begin{equation}\label{unrectifiable}\text{$\mu(f(E))=0$ whenever $f:E\subset\RR^{m-1}\rightarrow\XX$ is Lipschitz; }\end{equation}
\begin{equation}\label{m-rectifiable}\text{there exist Lipschitz maps $f_i:E_i\subset\RR^m\rightarrow\XX$ such that }\mu\left(\XX\setminus \textstyle\bigcup_1^\infty f_i(E_i)\right)=0.\end{equation} When $\XX=\RR^d$ and $1\leq m\leq d-1$, we also know that \begin{equation}\label{bi-lip-vanish}\text{$\mu(g(\RR^m))=0$ whenever $g:\RR^m\rightarrow \RR^d$ is a bi-Lipschitz embedding.}\end{equation}
\end{theorem}

The existence of a Radon measure $\mu$ on $\RR^d$ satisfying \eqref{doubling}, \eqref{m-rectifiable}, and \eqref{bi-lip-vanish} with $m=1$ is due to Garnett, Killip, and Schul \cite{GKS} (see the sentence ``In closing...'' on p.~1678). There is vast ocean between the cases $m=1$ and $m\geq 2$ insofar as a simple metric characterization\footnote{Wa\.zewski's Theorem: Let $\XX$ be a metric space and let $\Gamma\subset \XX$ be nonempty. Then $\Gamma=f([0,1])$ for some Lipschitz map $f:[0,1]\rightarrow\XX$ if and only if $\Gamma$ is compact, connected, and $\Haus^1(\Gamma)<\infty$. See \cite{AO-curves}.} of Lipschitz curves has been known since the 1920s, but no such result is available for higher-dimensional Lipschitz images. The best parameterization method for $m$-dimensional surfaces currently available when $m\geq 2$ is David and Toro's bi-Lipschitz variant of the Reifenberg algorithm \cite{DT,ENV-Banach}, but a bi-Lipschitz technique is useless for proving Theorem \ref{t:main} because \eqref{bi-lip-vanish} holds for any measure on $\RR^d$ satisfying the doubling property \eqref{doubling}. Property \eqref{haus-dim-s} is equivalent to the statement that there exists a Borel set $F\subset\XX$ of Hausdorff dimension $s_H$ such that $\mu(\XX\setminus F)=0$ and $\mu(E)=0$ whenever $E\subset\XX$ is a Borel set of Hausdorff dimension less than $s_H$. Property \eqref{pack-dim-s} is equivalent to the same assertion with Hausdorff dimension replaced by packing dimension (see \S\ref{ss:dimension}). Since Lipschitz maps do not increase packing dimension, \eqref{pack-dim-s} implies \eqref{unrectifiable} when $s_P>m-1$.

\begin{example}\label{big-list-of-examples} We emphasize that Theorem \ref{t:main}
makes no assumptions on the connectedness properties of the metric space. \begin{enumerate}
\item There exist doubling measures $\mu$ on $\RR^3$ of Hausdorff dimension $s_H=0.0001$ and packing dimension $s_P=1.9999$ that are 2-rectifiable and purely 1-unrectifiable. In fact, the measures can take the form of generalized Bernoulli products (see \S\ref{s:bernoulli}).
\item Any compact self-similar set of Hausdorff dimension $q$ in $\RR^d$ that satisfies the open set condition is Ahlfors $q$-regular and supports a $\lceil q\rceil$-rectifiable doubling measure that is purely $(\lceil q\rceil-1)$-unrectifiable. This large family of examples includes Cantor sets, which are totally disconnected.
\item The Koch snowflake curve in $\RR^2$ contains no non-trivial rectifiable subcurves, but is Ahlfors $\log_3(4)$-regular. Thus, the snowflake curve supports $1$-rectifiable doubling measures of Hausdorff and packing dimension $1-\epsilon$ for any $\epsilon>0$.
\item When $I=[0,1]^m$ is equipped with the \emph{snowflake metric} $d(x,y)=|x-y|^{m/s}$ for some $s>m$, the space $I$ is Ahlfors $s$-regular and $\Haus^s\res I$ is purely $m$-unrectifiable (because $s>m$). Nevertheless, $I$ supports an $m$-rectifiable doubling measure that is purely $(m-1)$-unrectifiable.
\item The first Heisenberg group $\mathbb{H}^1$ is a nonabelian Carnot group that is topologically equivalent to $\RR^3$, but equipped with a metric so that $\mathbb{H}^1$ has Hausdorff dimension $4$ and is Ahlfors $4$-regular. By \cite[Theorem 7.2]{AK}, the Hausdorff measures $\Haus^m\res\HH^1$ are purely $m$-unrectifiable for all $m\in\{2,3,4\}$. For further results on non-embedding of $\RR^m$ into Heisenberg groups, see e.g.~\cite{HS-Gromov-nonembedding} and the references therein. Even so, for all $m\in\{2,3,4\}$ and $s<m$, there exist doubling measures $\mu$ on $\mathbb{H}^1$ and Lipschitz maps $f:E\subset\RR^m\rightarrow\HH^1$ such that $\mu\ll\Haus^{s-\epsilon}$ for all $\epsilon>0$, $\dim_H f(E)=s$, and $\mu(f(E))>0$. That is, doubling measures on $\HH^1$ can charge Lipschitz images of Euclidean spaces of almost maximal dimension.
\item Let $\XX=L\cup Q$ be the union of a line segment $L=[0,1]\times\{0\}^2$ and a cube $Q=[1,2]\times[0,1]^2$, equipped with the subspace metric from $\RR^3$. While the space $\XX$ is not Ahlfors $q$-regular for any $q$, it is complete, doubling, and $\dim_H\XX=3$. Since $Q$ is Ahlfors $3$-regular, there exists a doubling measure $\nu$ with $\spt\nu= Q$ that is $2$-rectifiable and purely 1-unrectifiable. The measure $\mu:=\nu+\Haus^1\res L$ has $\spt\mu=\XX$, $\mu$ is 2-rectifiable, and $\mu$ is not $1$-rectifiable. However, $\mu$ is not purely $1$-unrectifiable, because $\mu\res L$ is (trivially!)~$1$-rectifiable.
\end{enumerate}\end{example}

The following seems plausible, but is beyond the scope of the present paper. Also see Conjecture \ref{all-d} for a related open problem.

\begin{conjecture} Theorem \ref{t:main} also holds when $s_P=m-1$, $s_P=m$, or $s_P=q$.\end{conjecture}

Our construction of the Lipschitz maps from (subsets of) $\RR^m$ into metric spaces $\XX$ is informed by the following simple observation about square packings.

\begin{lemma}\label{root-packing} Let $m,k\geq 2$ be integers. Any collection of $k^m$ axis-parallel cubes in $\RR^m$ of descending side lengths $s_0\geq s_1\geq \cdots \geq s_k\geq s_{k+1}\geq\cdots \geq s_{k^m-1}$ can be packed inside an axis-parallel cube of side length \begin{equation}\label{sum-law} s=s_{0^m}+s_{1^m}+s_{2^m}+\cdots+s_{(k-1)^m}.\end{equation} When $k=2$, this is the best possible bound, independent of the values of $s_2,\dots,s_{2^m-1}$.
\end{lemma}

\begin{figure}\begin{center}\includegraphics[width=.3\textwidth]{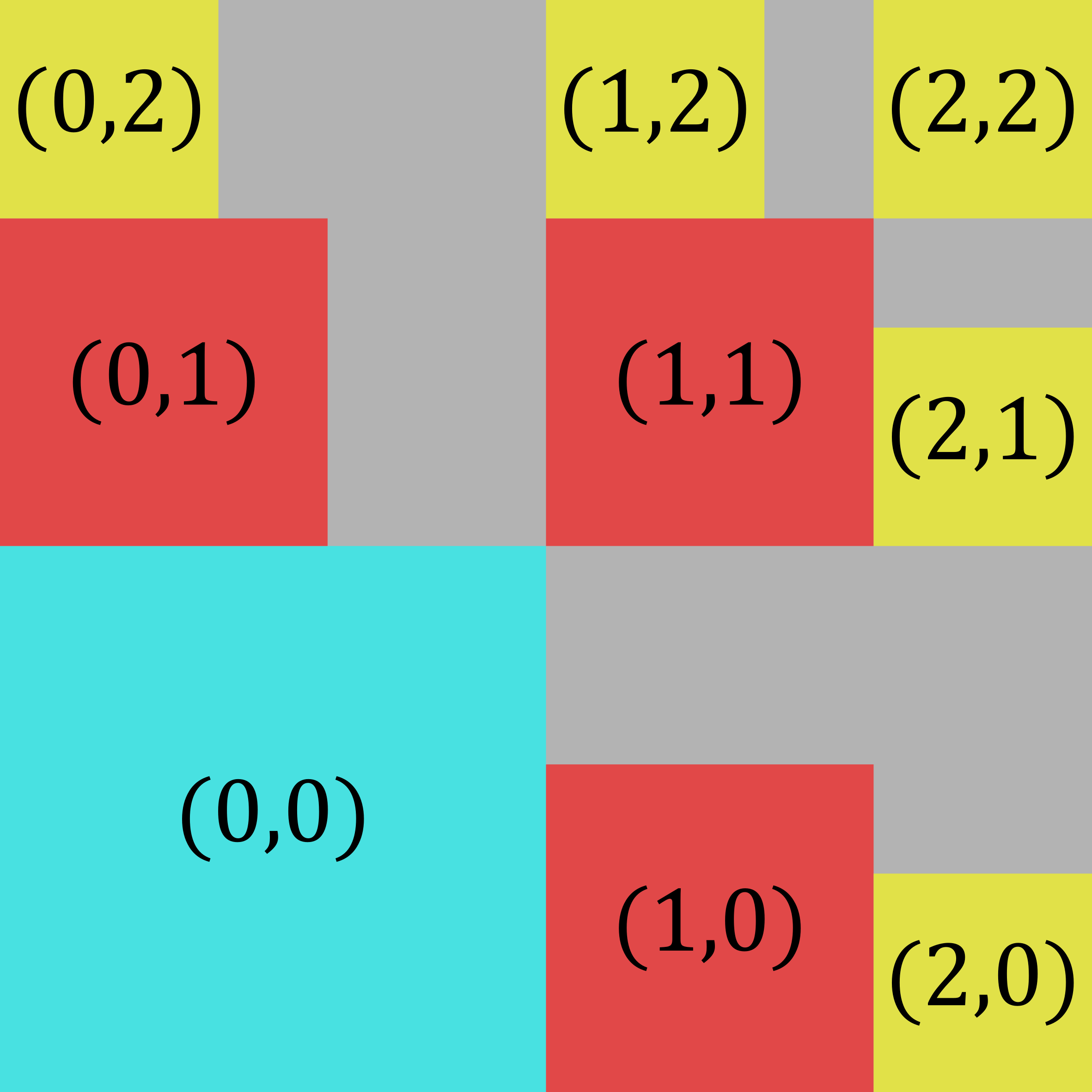}\end{center}\caption{The $3\times 3$ grid in the proof of Lemma \ref{root-packing} when $m=2$, $k=3$, and $s_0>s_1=s_2=s_3>s_4=s_5=s_6=s_7=s_8$.}\label{fig:grid}\end{figure}

\begin{proof} Let $Q_0,\dots, Q_{k^m-1}$ denote the cubes in the hypothesis. Create a $\underbrace{k\times\cdots\times k}_m$ grid of auxiliary cubes indexed by tuples $\{0,\dots,k-1\}^m$, where the cube in position $(i_1,\dots,i_m)$ has side length $s_{\mathbf{i}(i_1,\dots,i_m)}$, \begin{equation}\mathbf{i}(i_1,\dots,i_m):=\max\{i_1,\dots,i_m\}^m.\end{equation} See Figure \ref{fig:grid}. Inside of the grid, there are $(j+1)^m-j^m$ cubes of side length $s_{j^m}$ for each $0\leq j\leq k-1$. Because we have $s_{j^m}\geq s_{j^m+k}$ for all $k\geq 0$, the $(j+1)^m-j^m$ original cubes $Q_{j^m},\dots, Q_{(j+1)^m-1}$ may be arranged in one-to-one fashion to sit inside of the $(j+1)^m-j^m$ auxiliary cubes with $\mathbf{i}(i_1,\dots,i_m)=j^m$. By design, the grid of auxiliary cubes pack inside a cube of side length given by \eqref{sum-law}. Thus, so do the original cubes.

For the final claim, simply note that the side length of any cube $Q$ in $\RR^m$ that contains both $Q_0$ and $Q_1$ is at least $s_0+s_1$.
\end{proof}

\begin{corollary}\label{c-packing} It is possible to pack between $(k-1)^m+1$ and $k^m$ cubes in $\RR^m$ of equal side length $s$ into a cube of side length \begin{equation}\label{root-law} ks=\lceil\text{$m$-th root of number of cubes}\rceil\cdot \text{side length of a cube}.\end{equation}\end{corollary}

\begin{remark} By now classical results of Moon and Moser \cite{Moon-Moser} ($m=2$) and Meir and Moser \cite{Meir-Moser} ($m\geq 3$), any countable set of cubes in $\RR^m$ of total volume $V$ can be packed inside a cube of total volume $2^{m-1}V$. Even though they are simple, the preceding results indicate that total area of a list of squares is not a useful quantity for determining an \emph{optimal} square packing. See Figure \ref{fig:packing}.%
\begin{figure}\begin{center}\includegraphics[width=.6\textwidth]{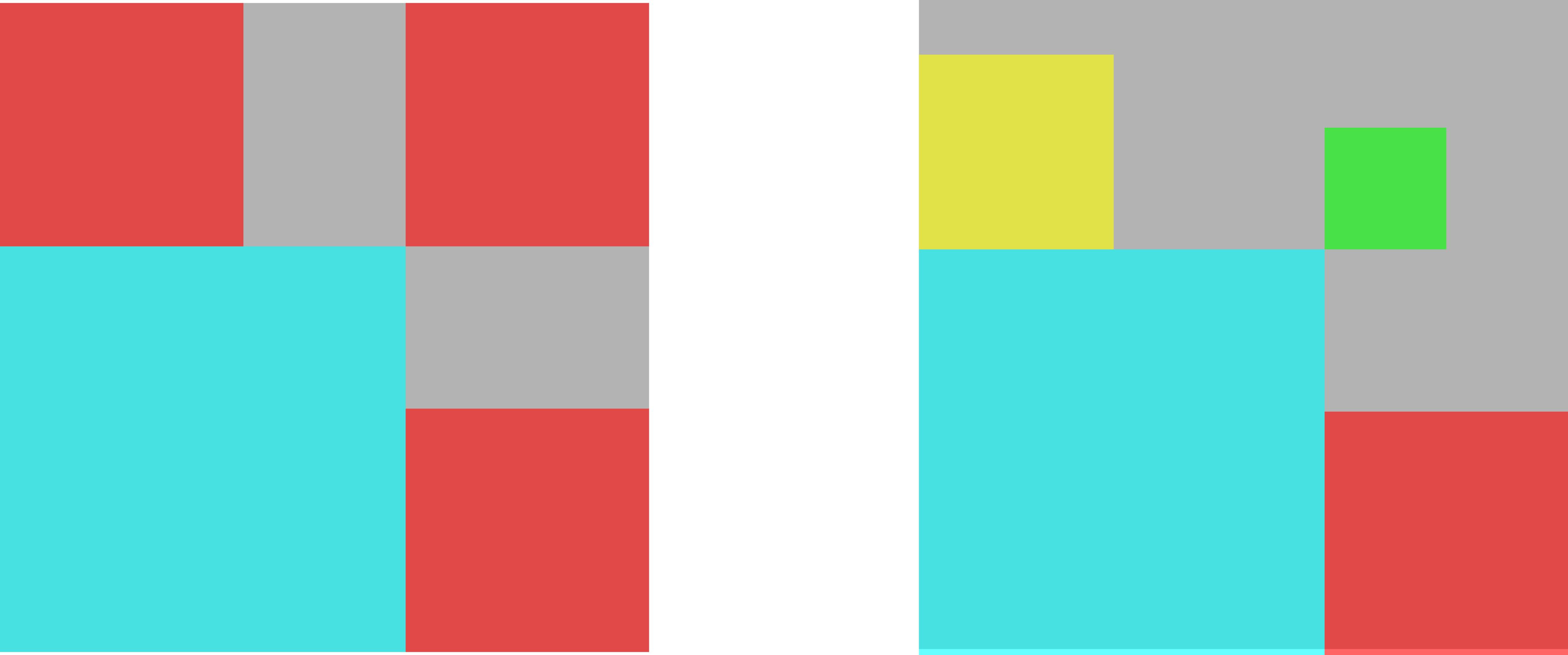}\end{center}\caption{Left: squares with side lengths $s_0>s_1=s_2=s_3$. Right: squares with side lengths $s_0>s_1>s_2'>s_3'$. Both sets of squares fit inside of a square of side length $s=s_0+s_1$. Adjusting the side lengths so that $s_1\approx s_0$ and $s_2',s_3'\ll s_0$, the ratio of the total area of the four squares on the left to the total area of the corresponding squares on the right can be made arbitrarily close to 2. Thus, scaling the picture on the right, there are lists of squares with the same total area, but different optimal packings.}\label{fig:packing}\end{figure}%
\end{remark}

\begin{remark}\label{r:tree-square} More generally, consider the following \emph{2-dimensional embedding problem} for \emph{trees of sets} in a metric space (there are a variety of possible definitions, see e.g.~\cite[\S2]{BV} or Definition \ref{def:tree} in the next section):
Given a tree of sets $\mathcal{T}$ in a metric space $\XX$, build (if it is possible to do so) a combinatorially equivalent tree $\mathcal{S}$ of nested squares in the plane and a map $Q\in\mathcal{T} \overset{S}{\mapsto} S_Q\in\mathcal{S}$ such that (i) $S_{R}\subset S_{Q}$ whenever $R$ is a descendant of $Q$ in $\mathcal{T}$ and (ii) $\diam S_Q\geq \diam Q$ for all $Q\in\mathcal{T}$. Furthermore, assuming that at least one solution exists, minimize the side length of $S_{\Top(\mathcal{T})}$.

Based on \cite[\S3]{BS2}, which handled  a related 1-dimensional problem, one might naively guess that $\sum_{Q\in\mathcal{T}}(\diam Q)^2<\infty$ implies existence of the tree $\mathcal{S}$ and map $S$.\footnote{When $\XX$ is locally quasiconvex, we now know that $\sum_{Q\in\mathcal{T}}(\diam Q)^2<\infty$ implies there exists a H\"older continuous map $f:[0,1]\rightarrow \XX$ with $|f(x)-f(y)|\leq H|x-y|^{1/2}$ for all $x,y\in[0,1]$ such that $f([0,1])$ contains $\leaves(\mathcal{T})$; see \cite{BV,Balogh-Zust}. For further related results, see \cite{AS-TST,BNV,Hyde-TST}.} However, the inadequacy of area for optimal square packings gives us a clear reason why this cannot be the case. A better candidate for a sufficient test for existence based on Corollary \ref{c-packing} appears to be finiteness of the maximal \emph{total diameter} of sets in a \emph{subtree} formed by keeping only \emph{square root many} children of each set in the tree (rounded up). 

For a concrete example, let us construct a tree $\mathcal{T}=\bigcup_{n=0}^\infty\mathcal{T}_n$ of axis-parallel cubes in $\RR^3$ as follows. Initialize $\mathcal{T}_0=\{[0,1]^3\}$. Assume that $\mathcal{T}_{n-1}$ has been defined for some $n\geq 1$. For each $Q\in\mathcal{T}_{n-1}$, include 9 subcubes $Q_1,\dots, Q_9$ of $Q$ of side length $s_n=\frac{1}{n}3^{-n}$ in the set $\mathcal{T}_n$ (8 subcubes in the corners, 1 subcube in the center). In the $n$-th level $\mathcal{T}_n$ of the tree, there are $9^n$ cubes of diameter $\sqrt{3}s_n$, where $s_0=1$. On the one hand, \begin{equation}\label{converging-sum}\sum_{Q\in\mathcal{T}}(\diam Q)^2 = 3+\sum_{n=1}^\infty 9^n\cdot 3s_n^2=3+\sum_{n=1}^\infty\frac{3}{n^2}<\infty.\end{equation} Thus, the Cantor set $E_1:=\bigcap_{n=0}^\infty \bigcup_{Q\in\mathcal{T}_n}Q$ of the leaves of the tree $\mathcal{T}$ has $\Haus^2(E_1)=0$. Moreover, it can be shown that $E_1$ has Hausdorff dimension $2$. On the other hand, suppose that $\mathcal{T}^{1/2}$ is any subtree of $\mathcal{T}$ such that the number of children of a cube in $\mathcal{T}^{1/2}$ is the square root of the number of children of that cube in $\mathcal{T}$. Then  \begin{equation}\label{diverging-sum}
\sum_{Q\in\mathcal{T}^{1/2}} \diam Q \geq \sum_{n=1}^\infty \sqrt{9^n}\cdot\sqrt{3} s_{n} = \sum_{n=1}^\infty \frac{\sqrt{3}}{n}=\infty. \end{equation} Appealing to Corollary \ref{c-packing}, we see that $\mathcal{T}$ cannot be represented as a tree $\mathcal{S}$ of nested squares with $\diam S_Q\geq \diam Q$ for all $Q\in\mathcal{T}$, because the maximal square $S_{[0,1]^3}$ in $\mathcal{S}$ would need to have infinite side length by \eqref{diverging-sum}.\end{remark}

\begin{remark}\label{r:alberti-csornyei} The Cantor set $E_1\subset\RR^3$ described in the previous remark is not contained in a Lipschitz image of $\RR^2$; a robust proof of this fact was communicated to the authors by G.~Alberti and M.~Cs\"ornyei in 2019.\end{remark}

\begin{remark}\label{ex:alpha} Let $\alpha>1$. Modify the tree in Remark \ref{r:tree-square} so that the cubes in $\mathcal{T}_n$ have side length $s_n=\frac{1}{n^\alpha}3^{-n}$ when $n\geq 1$. Each cube $Q\in\mathcal{T}_n$ has $N_n=9$ children and diameter $D_n=\sqrt{3}s_n$. Since $\alpha>1$, it follows that \begin{equation}S_\alpha:=\sum_{j=0}^\infty \left(\prod_{i=0}^{j} \lceil N_i^{1/2}\rceil \right) D_j= \sum_{j=0}^\infty 3^{j+1}D_j=3\sqrt{3}+3\sqrt{3}\sum_{j=1}^\infty \frac{1}{n^{\alpha}}<\infty.\end{equation} Therefore, by Theorem \ref{t:pack} / Corollary \ref{c:pack}, the Cantor set $E_\alpha=\bigcap_{n=0}^\infty \bigcup\mathcal{T}_n$ is contained in the image of an $S_\alpha$-Lipschitz map $f:[0,1]^2\rightarrow\RR^3$. An example of this kind was found by the first author and V.~Vellis in 2019. It helped lead to the results in \S2.\end{remark}

\begin{remark}Each of the Cantor sets $E_1$ and $E_\alpha$ ($\alpha>1$) are $\Haus^2$ null sets of Hausdorff dimension 2. One possible interpretation of the existence / non-existence of Lipschitz maps is that the sets $E_\alpha$ are (distorted) copies of null sets from $\RR^2$ inside of $\RR^3$, whereas the set $E_1$ is a ``new'' 2-dimensional null set in $\RR^3$ that does not exist in $\RR^2$.  \end{remark}

The rest of the paper is organized as follows. In \S\ref{s:Lipschitz}, we show how to use Lemma \ref{root-packing} to build Lipschitz images of $[0,1]^m$ containing the leaves of a tree of sets in a metric space. Simple applications to measures with positive lower density and finite upper density and to sets with small Assouad dimension are given in \S\ref{s:density}. The second half of the paper is devoted to the proof of the main theorem. In \S\ref{s:background}, we provide necessary background on the geometry of metric spaces and dimension of measures. In \S\ref{s:bernoulli}, we define and establish basic estimates for a family of quasi-Bernoulli measures $\muse$ on a complete Ahlfors $q$-regular metric space $\XX$, where $\mathbf{s}=(s_k)_{k\geq 1}$ is a sequence of ``target dimensions''. These measures are variants of the classic Bernoulli products on $[0,1]$. When $s_*=\inf_{k\geq 1}s_k>0$, the measure $\muse$ is doubling, and when $s=\lim_{k\rightarrow\infty} s_k< q$, the measure $\muse$ has exact Hausdorff and packing dimension $s$. In \S\ref{s:rect}, we record the proof of Theorem \ref{t:main}. In particular, we use the ``square packing construction'' of Lipschitz maps to show that when $s<m$, the measure $\muse$ is $m$-rectifiable.

\begin{remark}[prevalent notation] Throughout the paper, we use the letters $i,j,k,l,n$ interchangeably for indexing countable families, but reserve the letter $m$ for the dimension of Euclidean cube packings or the dimension of Euclidean space in the domain of Lipschitz maps $f:E\subset\RR^m\rightarrow\XX$ that appear in the definition of $m$-rectifiable measures. We always write $d$ for the dimension of Euclidean space in the event of a Euclidean codomain $\XX=\RR^d$. The letter $b$ is used exclusively for the scaling factor $b>1$ in a family of metric $b$-adic cubes and the letter $q$ is used exclusively for the dimension of an Ahlfors regular metric space. The letter $s$ generally refers to the dimension of a Hausdorff measure $\Haus^s$ or packing measure $\Pack^s$, but is sometimes used for side length of a square or cube. Greek letters such as $\delta,\epsilon,\tau$ represent errors or small parameters and usually take values in the range $(0,1)$, except for $\mu$ and $\nu$, which are reserved for measures. We write $c_{p_1,p_2,\dots}$ and $C_{p_1,p_2,\dots}$ to denote indeterminate positive and finite constants with values that can be bounded above and below using the listed parameters $p_1,p_2,\dots$. As is nowadays common, the notation $x\lesssim_{p_1,p_2,\dots} y$ is short hand for $x\leq C_{p_1,p_2,\dots} y$, when we don't need to manipulate $C_{p_1,p_2,\dots}$ in subsequent expressions. The notation $x\ll y$ or $y\gg x$ is sometimes used for emphasis and is meant to be read as ``$x$ is much smaller than $y$'' or ``$y$ is much larger than $x$''. We typically denote the underlying distance between points $x$ and $y$ in a metric space by $|x-y|$. The \emph{diameter} of a nonempty set $E$ in a metric space is $\diam E:=\sup\{|x-y|:x,y\in E\}$. The \emph{gap} between nonempty sets $A$ and $B$ in a metric space is $\gap(A,B):=\inf\{|x-y|:x\in A,y\in B\}$. (This terminology comes from variational analysis. The notation $\dist(A,B)$ is used more often in the literature, but ``distance'' is problematic, since $\dist(A,C)\leq \dist(A,B)+\dist(B,C)$ usually fails.) Finally, just in case it is not familiar to the reader, we mention that the \emph{ceiling} of a real number $x$ is $\lceil x\rceil := \inf\{n\in\ZZ:x\leq n\}$. \end{remark}

\section{Square packings and Lipschitz maps}\label{s:Lipschitz}

Because it requires no more effort, we adopt a very weak definition of a tree of sets in a metric space that allows for overlap and repetition. The only coherence condition is that each set is contained in its parent.

\begin{definition}\label{def:tree} Let $\XX$ be a metric space. We call $\mathcal{T}=\bigsqcup_{j=0}^\infty\mathcal{T}_j$ a \emph{tree of sets} in $\XX$ if \begin{itemize}
\item $\mathcal{T}_j$ is a finite multiset\footnote{A finite multiset is a finite unordered list with repetition allowed. For example, $\{1,2,2\}$ and $\{1,1,2,2,2\}$ are multisets with 3 and 5 elements, respectively. Their disjoint union $\{1,2,2\}\sqcup \{1,1,2,2,2\}$ is the multiset $\{1,1,1,2,2,2,2,2\}$ with 8 elements. A function between multisets is a function between sets formed by assigning any repeated elements a different color. For example, we could define an injective function $f:\{1,2,2\}\rightarrow\{1,1,2,2,2,2\}$ by forming sets $\{1,2,\color{PTblue}2\color{black}\}$ and $\{1,\color{PTblue}1\color{black},2,\color{PTblue}2\color{black},\color{PTmagenta}2\color{black},\color{PTteal}2\color{black}\}$ and defining $f(1)=2$, $f(2)=\color{PTblue}2\color{black}$, and $f(\color{PTblue}2\color{black})=\color{PTmagenta}2\color{black}$.} of nonempty subsets of $\XX$ for all $j\geq 0$;
\item $\#\mathcal{T}_0=1$ (counting multiplicity); and,
\item there is a function $\uparrow:\bigsqcup_{j=1}^\infty \mathcal{T}_j\rightarrow \mathcal{T}$ such that $Q\subset Q^\uparrow\in\mathcal{T}_{j-1}$ for all $j\geq 1$ and $Q\in\mathcal{T}_j$.
\end{itemize} We call $Q^\uparrow$ the \emph{parent} of $Q$ and call $Q$ a \emph{child} of $Q^\uparrow$. For all $j\geq 0$ and $Q\in\mathcal{T}_j$, we let $\Child(Q):=\{R\in\mathcal{T}_{j+1}:Q=R^\uparrow\}$ denote the \emph{set of children} of $Q$. We denote the unique set in $\mathcal{T}_0$ by $\Top(\mathcal{T})$. An \emph{infinite branch} in $\mathcal{T}$ is a sequence of sets $Q_j\in\mathcal{T}_j$ such that $Q_{j+1}\in\Child(Q_j)$ for all $j\geq 0$. Finally, the \emph{set of leaves} of $\mathcal{T}$ is defined by \begin{equation}\leaves(\mathcal{T}):=\bigcap_{j=0}^\infty \bigcup_{Q\in\mathcal{T}_j} Q.\end{equation}
\end{definition}

\begin{remark}Let $\mathcal{T}$ be a tree of sets in a metric space $\XX$. If every set $Q\in\mathcal{T}$ is closed, then $\leaves(\mathcal{T})$ is closed. If every set $Q\in\mathcal{T}$ is Borel, then $\leaves(\mathcal{T})$ is Borel.
\end{remark}

Recall that a map $f$ between metric spaces is $L$-Lipschitz if $|f(x)-f(y)|\leq L|x-y|$ for all $x$ and $y$ in the domain of $f$. Here and throughout the paper, we use the convention that $|\cdot-\cdot|$ denotes distance between points in the appropriate metric. We let $\ell^m_p$ denote $\RR^m$ equipped with the $\ell_p$ norm. In particular, $\ell^m_2$  and $\ell^m_\infty$ are equipped with the standard Euclidean norm and the supremum norm, respectively.

\begin{lemma}[square packing construction of Lipschitz maps] \label{l:pack} Let $\mathcal{T}=\bigsqcup_{j=0}^\infty \mathcal{T}_j$ be a tree of sets in a metric space $\XX$. For each $j\geq 0$, assign \begin{equation}\label{N-def} N_j:=\max_{Q\in\mathcal{T}_j}\#\Child(Q)\quad \text{and}\quad D_j:=\max_{Q\in\mathcal{T}_j}\,\diam Q.\end{equation}  Let $m,l\geq 1$ be integers and suppose that $\mathcal{T}_l\neq\emptyset$. Compute the finite quantity \begin{equation}\label{S-condition} s:=\sum_{j=0}^{l-1}\left(\prod_{i=0}^{j-1} \lceil N_i^{1/m}\rceil\right) (\lceil N_j^{1/m}\rceil-1)D_j.\end{equation} (When $j=0$, $\prod_{i=0}^{j-1} \lceil N_i^{1/m}\rceil=1$.) For any multiset $F=\{x_Q\in Q:Q\in\mathcal{T}_l\}$, there exists a set $E\subset \ell_\infty^m\cap [0,s]^m$ with $\#E=\#F<\infty$ and a bijective 1-Lipschitz map $f:E\rightarrow F$.
\end{lemma}

\begin{figure}[t]\begin{center}\includegraphics[width=.45\textwidth]{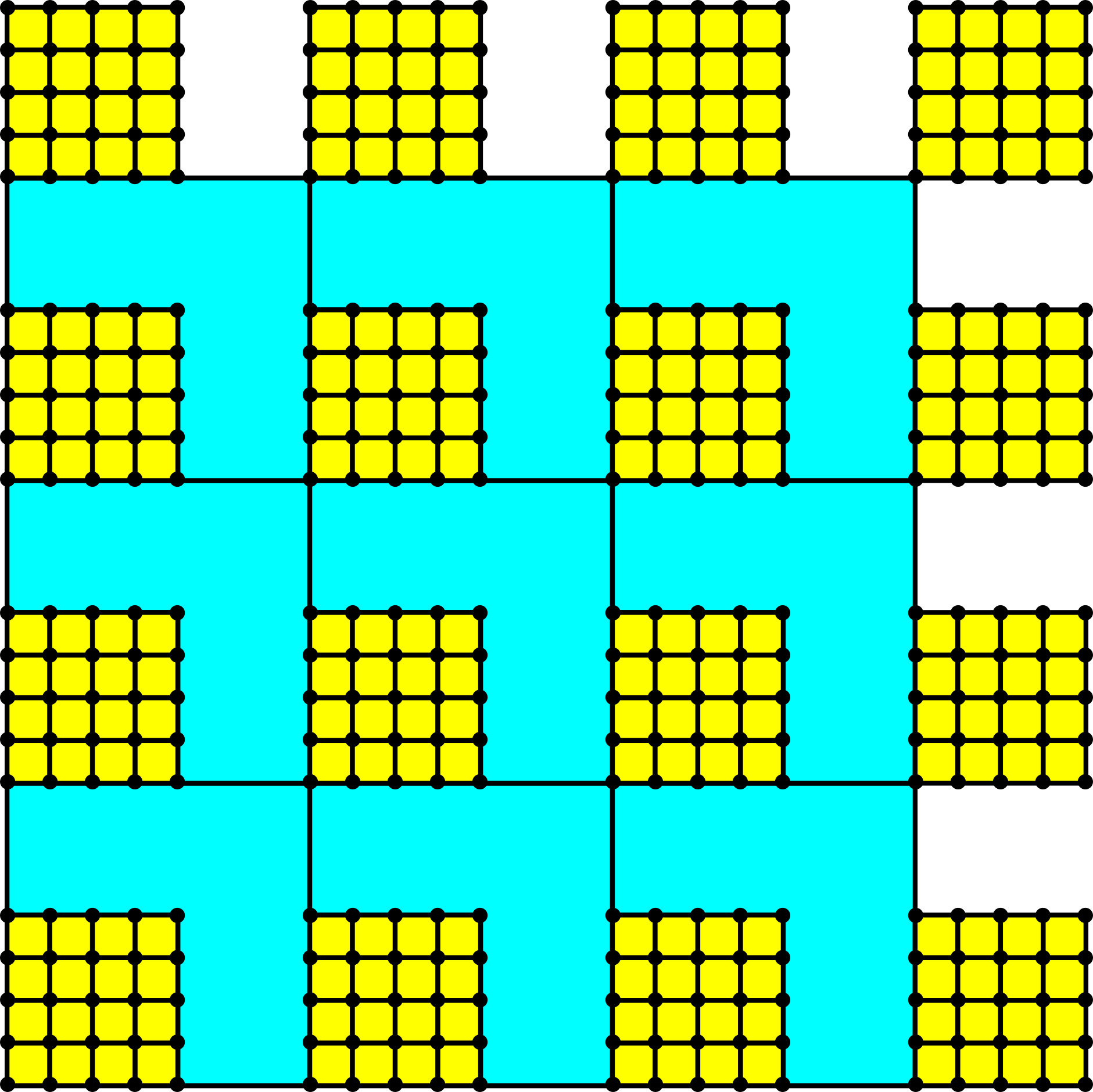}\end{center}\caption{Illustration of the domain of the Lipschitz map $f$ in the ``square packing construction'' with dimension $m=2$, tree depth $l=2$, and the maximal number of children of sets in $\mathcal{T}_0$ and $\mathcal{T}_1$ given by $N_0=16$ and $N_1=25$, respectively. The side length of each of yellow square is $D_1$ and the side length of each ``block'' of yellow squares is $(\lceil N_1^{1/2}\rceil -1)D_1$. The side length of each blue square is equal to the side length of a yellow block plus $D_0$. All together, the domain sits inside a square of side length $(\lceil N_0^{1/2}\rceil-1)D_0+\lceil N_0^{1/2}\rceil(\lceil N_1^{1/2}\rceil-1) D_1$.}\label{fig:two-sizes}\end{figure}

\begin{proof} Let $\mathcal{T}$, $m$, $l$, and $F$ be given. We want to build a 1-Lipschitz map $f$ whose image is $F$. Fix any coloring on $F$. For each point $x'\in F$, we must decide how to place a point $x=g(x')$ in $\RR^m$ such that for every pair of points $x',y'\in F$ on the image side of $f$, we have $|x-y|\geq |x'-y'|$ on the domain side of $f$. (That is, $g$ must be \emph{distance non-decreasing}.) The familial relationships in $\mathcal{T}$, the quantities $N_j$ and $D_j$, and Lemma \ref{root-packing} will tell us one way that we can accomplish this. Further, the assignment $g$ from points in $F$ to points in $\RR^m$ will be one-to-one. Once we have placed a point $x$ in $\RR^m$ for each $x'\in F$, we simply assign $E:=g(F)$ and define $f:E\rightarrow F$ to be the inverse of the bijection $g:F\rightarrow E$. The map $f$ is 1-Lipschitz, because $|f(x)-f(y)|=|x'-y'|\leq |x-y|$ by the stipulation above.

The description of the assignment $g$ is recursive. Suppose that we know how to do the construction for trees of depth $l-1$. Let $\mathcal{T}_l$ be nonempty and let $F=\{x_Q:Q\in\mathcal{T}_l\}$. Then the truncated tree $\bigsqcup_{j=0}^l\mathcal{T}_j$ is a disjoint union of $N_0=\#\Child(\Top(\mathcal{T}_0))=\#\mathcal{T}_1$ trees $\mathcal{T}^{1},\cdots,\mathcal{T}^{N_0}$ of depth $l-1$, where $\{\Top(\mathcal{T}^i):1\leq i\leq N_0\}=\Child(\Top(\mathcal{T}_0))=\mathcal{T}_1$. Notice that the $j$-th level of any $\mathcal{T}^i$ is made up of sets belonging to the $(j+1)$-st level of $\mathcal{T}$. Hence $N_j(\mathcal{T}^i)\leq N_{j+1}(\mathcal{T})$ and $D_j(\mathcal{T}^i)\leq D_{j+1}(\mathcal{T})$ for all $0\leq j\leq l-2$. Write $F=\bigsqcup_{i=1}^{N_0}F^i$, where $F^i=\{x_Q:Q\in\mathcal{T}_l\text{ descends from} \Top(\mathcal{T}^i)\}$. Since we know how to do the construction for trees of depth $l-1$, for each index $1\leq i\leq N_0$, we can find a set $E_i\subset \RR^m$ and an injective 1-Lipschitz map $f_i:E_i\rightarrow \XX$ such that $f(E_i)=F_i$ and $E_i$ is contained in some cube $S_i$ in $\RR^m$ of side length at most $s_{l-1}$ for some number $s_{l-1}>0$ depending only on $m$ and $N_1,\dots,N_{l-1}$ and $D_1,\dots,D_{l-1}$. (Assuming momentarily that \eqref{S-condition} is correct, then \begin{equation}\label{previous-s} s_{l-1}:= \sum_{j=0}^{l-2}\left(\prod_{i=0}^{j-1} \lceil N_{i+1}^{1/m}\rceil\right) (\lceil N_{j+1}^{1/m}\rceil-1)D_{j+1}=\sum_{j=1}^{l-1}\left(\prod_{i=1}^{j-1}
\lceil N_{i}^{1/m}\rceil\right) (\lceil N_j^{1/m}\rceil -1)D_j\end{equation} will suffice.) By Lemma \ref{root-packing}, we can pack a collection of $\left(\lceil N_0^{1/m}\rceil-1\right)^m$ ``blue'' cubes of side length $(s_{l-1}+D_0)$ and $\lceil N_0^{1/m}\rceil^m-\left(\lceil N_0^{1/m}\rceil-1\right)^m$ ``red'' cubes of side length $s_{l-1}$ inside a cube of side length \begin{equation}\label{recursion-formula} s_l:=(\lceil N_0^{1/m}\rceil-1)(s_{l-1}+D_0)+s_{l-1}=\lceil N_0^{1/m}\rceil s_{l-1}+(\lceil N_0^{1/m}\rceil-1)D_0.\end{equation} Moreover, as in the proof of the lemma, we can arrange things so that the red cubes sit ``to the right'' of the blue cubes in each coordinate and any two distinct red cubes are separated by a distance at least the side length of a blue cube minus the side length of a red cube. See Figure \ref{fig:grid}, focusing only on the blue and red squares. To proceed, place translated copies $\vec S_i$ of the cubes $S_i$ (and hence translated copies $\vec E_i$ of the sets $E_i$) inside the collection of blue and red cubes in a one-to-one fashion. We can always do this, because $N_0\leq \lceil N_0^{1/m}\rceil^m$ and the side length of any $S_i$ is less than or equal to the side length of a blue or red cube. In Figure \ref{fig:two-sizes}, the cubes $\vec S_i$ are the ``yellow blocks''; the yellow blocks on the right and top of the figure cover the red cubes. When placing a cube $\vec S_i$ inside a blue cube, let's stipulate that we place $\vec S_i$ as far ``to the left'' as possible in each coordinate. This ensures than any two distinct $\vec S_i$ are separated by a distance at least $D_0$. We now define $E=\bigcup_{i=1}^{N_0} \vec E_i$ and define $f:E\rightarrow F$ by setting $f|_{\vec E_i}(\vec x)=f_i(x)$ for all $1\leq i\leq N_0$ and all $\vec x\in \vec E_i$, where $x$ is the unique point in $E_i$ corresponding to $\vec x$.

The map $f$ is injective (as a multiset map), because each $f_i$ is injective and their targets $F_i$ are disjoint inside of $F$ (with the fixed coloring). Next, let's check that $f$ is 1-Lipschitz.  Let $\vec x,\vec y\in E$. If $\vec x$ and $\vec y$ both belong to $\vec E_i$ for some $i$, then $|f(\vec x)-f(\vec y)|=|f_i(x)-f_i(y)|\leq |x-y|=|\vec x-\vec y|$, because $f_i$ is 1-Lipschitz and translation is an isometry. Suppose instead that $\vec x\in\vec E_i$ and $\vec y\in\vec E_j$ for some $i\neq j$. Then $|\vec x-\vec y|\geq D_0=\diam \Top(\mathcal{T})\geq |f(\vec x)-f(\vec y)|$. In both cases, we checked that $|f(\vec x)-f(\vec y)|\leq |\vec x-\vec y|$. Thus, $f$ is 1-Lipschitz, as claimed.

Solving the recurrence relation \eqref{recursion-formula} with initial condition $s_0=0$ yields \eqref{S-condition}. (If $l=0$, then $F=\{x_{\Top(\mathcal{T})}\}$ consists of a single point and the domain $E=\{0\}$ of $f$ is a ``cube'' of side length $s_0=0$. Note that $s_{l-1}=(\lceil N_1^{1/m}\rceil)s_{l-2}+(\lceil N_1^{1/m}\rceil-1) D_1$, etc.) Alternatively, one may verify \eqref{S-condition} using induction. Indeed, substituting \eqref{previous-s} into \eqref{recursion-formula}, we have \begin{align*}
s_l&= \lceil N_0^{1/m}\rceil s_{l-1}+(\lceil N_0^{1/m}\rceil-1) D_0\\
&=\left(\sum_{j=1}^{l-1}\left(\prod_{i=0}^{j-1}
\lceil N_{i}^{1/m}\rceil\right) (\lceil N_j^{1/m}\rceil -1)D_j\right)+(\lceil N_0^{1/m}\rceil-1)D_0\\
&=\sum_{j=0}^{l-1}\left(\prod_{i=0}^{j-1} \lceil N_i^{1/m}\rceil\right) (\lceil N_j^{1/m}\rceil-1)D_j. \qedhere\end{align*}\end{proof}

\begin{lemma}[each leaf sits at the end of a branch] \label{l:branch} Let $\mathcal{T}=\bigsqcup_{j=0}^\infty\mathcal{T}_j$ be a tree of sets in a metric space $\XX$. For each $j\geq 0$, let $D_j$ be given by \eqref{N-def}. If $\lim_{j\rightarrow\infty} D_j=0$, then for each $z\in\leaves(\mathcal{T})$, there exists an infinite branch $(Q_j^z)_{j=0}^\infty$ in $\mathcal{T}$ such that $\{z\}=\bigcap_{j=0}^\infty Q_j^z$. \end{lemma}

\begin{proof} Let $z\in\leaves(\mathcal{T})$. Since $\#\mathcal{T}_j<\infty$ for all $j\geq 0$, there exists an infinite branch $(Q_j^z)_{j=0}^\infty$ in $\mathcal{T}$ such that $z\in\bigcap_{j=0}^\infty Q_j^z$ by K\"onig's lemma; see e.g.~\cite{Konig-history} or \cite[p.~190]{DST}. Further, since every child is contained in its parent, $\diam\bigcap_{j=0}^\infty Q_j^z \leq \lim_{j\rightarrow\infty} D_j=0$.\end{proof}

\begin{theorem}[sufficient condition for the set of leaves to lie in a Lipschitz image] \label{t:pack} Let $\mathcal{T}=\bigsqcup_{j=0}^\infty\mathcal{T}_j$ be a tree of sets in a metric space $\XX$. For each $j\geq 0$, let $N_j$ and $D_j$ be given by \eqref{N-def}. If the closure of $\Top(\mathcal{T})$ is compact and \begin{equation}\label{root-condition} S:=\sum_{j=0}^\infty \left(\prod_{i=0}^j \lceil N_{j}^{1/m}\rceil \right) D_j<\infty,\end{equation} then there exists a compact set $E\subset \ell^m_\infty\cap [0,1]^m$ and a $S$-Lipschitz map $f:E\rightarrow\XX$ such that $f(E)$ contains $\leaves(\mathcal{T})$.\end{theorem}

\begin{proof} Replacing $\XX$ with $\overline{\Top(\mathcal{T})}$, we may assume without loss of generality that $\XX$ is compact. Further, we may assume that $\leaves(\mathcal{T})\neq\emptyset$, otherwise there is nothing to show. In particular, $\mathcal{T}_l$ is nonempty for all $l\geq 0$. Note that the quantity $s=s(l)$ in \eqref{S-condition} is bounded above by $S$ in \eqref{root-condition} for all $l\geq 1$. For each $l\geq 1$, choose any set $F_l=\{x_Q:Q\in\mathcal{T}_l\text{ and }\Child(Q)\neq\emptyset\}$. After rescaling the domains of maps provided by Lemma \ref{l:pack}, for each $l\geq 1$, we can find a finite set $E_l\subset \ell^m_\infty\cap [0,1]^m$ and an $S$-Lipschitz map $f_l:E_l\rightarrow \XX$ such that $f_l(E_l)=F_l$.

To proceed, we employ a variation of the proof of the Arzel\`a-Ascoli theorem for a sequence of uniformly equicontinuous functions with \emph{variable domains}. Let $\YY=[0,1]^m\times \XX$ be equipped with the product metric and let the space $\mathfrak{C}(\YY)$ of nonempty closed subsets of $\YY$ be equipped with the Hausdorff metric. Since $[0,1]^m$ and $\XX$ are compact, so are $\YY$ and $\mathfrak{C}(\YY)$; see e.g.~\cite{Beer}. The latter fact is often called \emph{Blaschke's selection theorem}. Thus, the sequence of graphs $\Gamma_l=\{(x,f_l(x)):x\in E_l\}$ in $\mathfrak{C}(\YY)$ have a subsequence $\Gamma_{l_i}$ that converge to some set $\Gamma$ in $\mathfrak{C}(\YY)$ as $i\rightarrow\infty$. We claim that $\Gamma$ is also a graph. Indeed, suppose that $(x,y)$ and $(x,y')$ belong to $\Gamma$. Then we can find sequences $x_i,x'_i\in E_{l_i}$ such that $(x_i,f_{l_i}(x_i))\rightarrow (x,y)$ and $(x'_i,f_{l_i}(x'_i))\rightarrow (x,y')$ in $\YY$ as $i\rightarrow\infty$. Hence, by the uniform Lipschitz condition, $$|y-y'|\leq \liminf_{i\rightarrow\infty} |f_{l_i}(x_i)-f_{l_i}(x'_i)| \leq \liminf_{i\rightarrow\infty} S|x_i-x'_i|=0.$$ Write $E=\{x:(x,y)\in\Gamma\}\subset[0,1]^m$ and define $f:E\rightarrow\XX$ according to the rule $(x,f(x))\in\Gamma$ for all $x\in E$. The domain $E$ is compact, because $\Gamma$ is compact. It is easy to see that $f$ is $S$-Lipschitz. Given $x,x'\in E$, there exist sequences $(x_i,f_{l_i}(x_i))\rightarrow (x,f(x))$ and $(x'_i,f_{l_i}(x'_i))\rightarrow (x',f(x'))$ in $\YY$ as $i\rightarrow\infty$. Therefore,  $$|f(x)-f(x')|=\lim_{i\rightarrow\infty} |f_{l_i}(x_i)-f_{l_i}(x'_i)|\leq \liminf_{i\rightarrow\infty} S|x_i-x'_i|=S|x-x'|.$$ Finally, the image set $f(E)$ contains $\leaves(\mathcal{T})$ by Lemma \ref{l:branch} and compactness of $\YY$. Indeed, we may use the lemma, because $\sum_{j=0}^\infty D_j\leq S<\infty$ implies $\lim_{j\rightarrow \infty} D_j=0$. Thus, given $z\in\leaves(\mathcal{T})$,  there exists an infinite branch $(Q_j^z)_{j=0}^\infty$ in $\mathcal{T}$ such that $\{z\}=\bigcap_{j=0}^\infty Q_j^z$. In particular, the points $z_i:=x_{Q^{z}_{l_i}}\in F_{l_i}$ converge to $z$ in $\XX$ as $i\rightarrow\infty$. For each $i$, choose any point $x_i\in E_{l_i}$ such that $f_{l_i}(x_i)=z_i$. Since we do not assume each $f_{l_i}$ is injective, there is no reason to suspect that the sequence $x_i$ converges. Nevertheless, by compactness of $\YY$, we may find a subsequence $(x_{i_j},z_{i_j})$ that converges to some point $(x,y)\in\YY$ as $j\rightarrow\infty$. Of course, $y=z$, since $z_{i_j}\rightarrow z$ as $j\rightarrow\infty$. Further, since $(x_{i_j},z_{i_j})\in\Gamma_{l_{i_j}}$ for all $j$ and $\Gamma_{l_{i_j}}\rightarrow \Gamma$ in $\mathfrak{C}(\YY)$ as $j\rightarrow\infty$, the limit  point $(x,z)\in\Gamma$. That is, $x\in E$ and $z=f(x)\in f(E)$.
\end{proof}

\begin{remark}[no rounding]\label{r:no-ceil} If $C=\sum_{i=0}^\infty N_i^{-1/m}<\infty$, then \begin{equation}S\leq e^C\sum_{j=0}^\infty \left(\prod_{i=0}^j N_{j}^{1/m} \right) D_j,\end{equation} because $\lceil x\rceil \leq (1+x^{-1})x$ for all $x>0$ and $\prod_{i=0}^\infty (1+N_i^{-1/m})\leq e^C$. Thus, to check the hypothesis of Theorem \ref{t:pack} on a tree, in which the maximum number of children $N_j\rightarrow\infty$ quickly as $j\rightarrow\infty$, we can effectively ignore the ceiling function in \eqref{root-condition}. We will use this observation in \S\ref{s:rect}.\end{remark}

\begin{remark}\label{r:other-p} Let $1\leq p<\infty$. Because $\|x\|_\infty\leq \|x\|_p$ for all $x\in\RR^m$, both Lemma \ref{l:pack} and Theorem \ref{t:pack} remain valid with the domain of the map $f$ in the conclusion replaced by a set $E\subset \ell_p^m\cap[0,s]^m$ or $E\subset\ell_p^m\cap[0,1]^m$, respectively. \end{remark}

\begin{corollary}[Lipschitz maps from $\ell_2^m$] \label{c:pack} In addition to the hypothesis of Lemma \ref{l:pack} or Theorem \ref{t:pack}, suppose that $\XX$ is a Hilbert space. In the first setting, there exists a 1-Lipschitz map $g:\ell_2^m\cap [0,s]^m\rightarrow \XX$ such that $g([0,s]^m)\supset F$. In the second setting, there exists a $S$-Lipschitz map $g:\ell_2^m\cap [0,1]^m\rightarrow\XX$ such that $g([0,1]^m)\supset\leaves(\mathcal{T})$.   \end{corollary}

\begin{proof} Let $f:E\rightarrow \XX$ be given by Lemma \ref{l:pack}, taking the domain $E\subset \ell_2^m\cap[0,S]^m$. Since $\XX$ is a Hilbert space, we may apply Kirzbraun's theorem \cite[2.10.43]{Federer} to extend $f$ to a 1-Lipschitz map $h:\ell_2^m\rightarrow\XX$. Then the restriction $g:=h|_{\ell_2^m\cap [0,S]^m}$ is 1-Lipschitz and $g([0,S]^m)=h([0,S]^m)\supset h(E)=f(E)=F$. A similar argument works in the setting of Theorem \ref{t:pack}.\end{proof}

\begin{remark} Using packings with rotated squares or cubes, it is possible to improve Corollary \ref{c:pack} in certain situations. For example, let $\mathcal{T}$ be a tree of sets in a metric space $\XX$ such that $\Top(\mathcal{T})$ is contained in a compact set and $N_j=5$ for all $j\geq 0$. Because $\lceil \sqrt{5}\rceil = 3$, Corollary \ref{c:pack} tells us that $\leaves(\mathcal{T})$ is contained in a Lipschitz image of a Euclidean square if \begin{equation}\sum_{j=0}^\infty 3^{j+1} D_j<\infty.\end{equation} However, allowing rotations, 5 unit squares in the Euclidean plane can be packed inside a square of side length $2+\frac{1}{\sqrt{2}}$ by arranging four axis-parallel squares around a square that is rotated by $45^\circ$. Thus, \emph{mutatis mutandis}, the proofs above show that $\leaves(\mathcal{T})$ sits inside a Lipschitz image of a Euclidean square if \begin{equation}\label{5-rate-2}\sum_{j=0}^\infty (2+\tfrac{1}{\sqrt{2}})^{j+1} D_j<\infty.\end{equation} See Figure \ref{fig:5-packing-2}. The problem of finding the minimal side length $s$ of a square containing $N$ unit squares (with rotations) is open except for some sporadic values of $N$. See the survey \cite{square-packing-survey} for an illuminating introduction.

\begin{figure}\begin{center}\includegraphics[width=.3\textwidth]{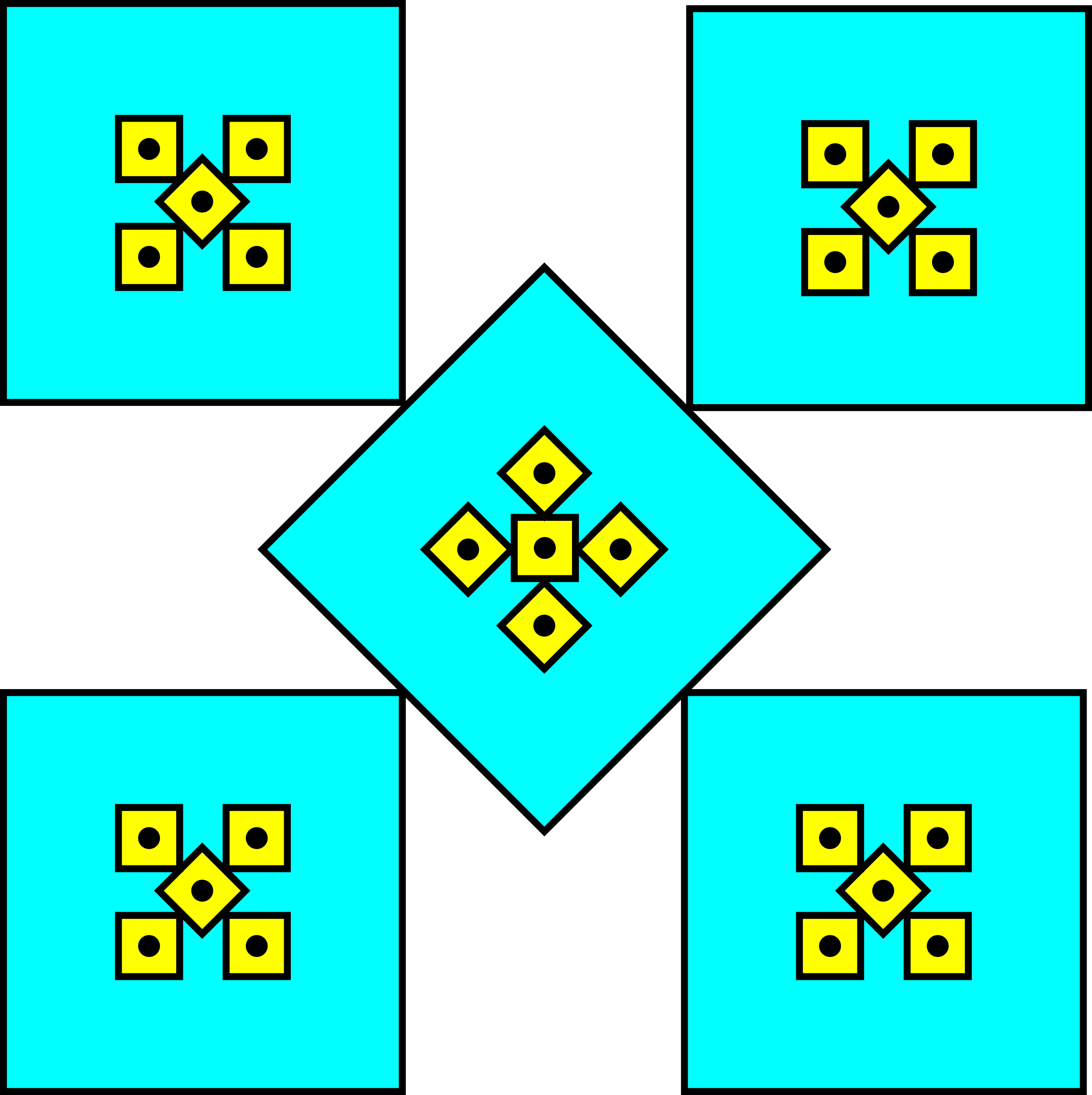}\end{center}\caption{Modified construction of the Lipschitz map $f$ in Lemma \ref{l:pack} with dimension $m=2$, tree depth $l=2$, and maximal number of children of sets in $\mathcal{T}_0$ and $\mathcal{T}_1$ given by $N_0=N_1=5$. The side length of each of yellow square is $D_1$ and the side length of each ``block'' of yellow squares is $(2+\frac{1}{\sqrt{2}})D_1$. The side length of each blue square is equal to the side length of a yellow block plus $D_0$. All together, the domain sits inside a square of side length $(2+\frac{1}{\sqrt{2}})D_0+(2+\frac{1}{\sqrt{2}})^2 D_1$.}\label{fig:5-packing-2}\end{figure}

A clever reader may also notice that $\leaves(\mathcal{T})$ is contained in a Lipschitz image of a Euclidean square when $N_j=5$ for all $j$ if  $\sum_{j=0}^\infty 5^{2j+1}D_{2j}<\infty$. (Skip odd generations.) This is a further improvement over \eqref{5-rate-2}, because $5<(2+\frac{1}{\sqrt{2}})^2$.
\end{remark}

\begin{remark}[square packings and H\"older maps] \label{r:holder} It is easy to modify the statements and proofs above to produce H\"older maps in place of Lipschitz maps. To produce maps satisfying $|f(x)-f(y)|\leq H|x-y|^{1/s}$ simply replace the quantity $D_j$ by $D_j^s$.\end{remark}

\section{Densities, Assouad dimension, and rectifiability}\label{s:density}

As an initial application of Theorem \ref{t:pack}, we extend Martin and Mattila's \cite[Theorem 4.1(1)]{MM1988} on the rectifiability of Hausdorff measures on $s$-sets in $\RR^d$ when $s<m$ (also see \cite[Theorem C]{BV}) to general measures on complete metric spaces. To build a tree, \cite{MM1988} and \cite{BV} each used the Besicovitch covering theorem, which is not available in a general metric space. We are able to avoid reliance on a covering theorem by utilizing $r$-nets, i.e.~maximal subsets of $r$-separated points.

\begin{theorem}\label{t:mm} Let $\mu$ be a finite Borel measure on a complete metric space $\XX$. For every integer $m\geq 1$ and real-valued dimension $s\in[0,m)$, the measure \begin{equation}\label{lu-density}\mu\res\left\{x\in\XX: 0<\liminf_{r\downarrow 0} \frac{\mu(B(x,r))}{r^s}\leq \limsup_{r\downarrow 0} \frac{\mu(B(x,r))}{r^s}<\infty\right\}\end{equation} is $m$-rectifiable. In fact, the set described in \eqref{lu-density} is contained in a countable union of images of Lipschitz maps of the form $f:E\subset[0,1]^m\rightarrow\XX$.
\end{theorem}

\begin{proof} For each $n\geq 2$, let $G_n:=\{x\in\XX: r^s/n\leq \mu(B(x,r))\leq nr^s\text{ for all }0<r<1/n\}$. Because the set of points with positive lower $s$-density and finite upper $s$-density can be written as $\bigcup_{n=2}^\infty G_n$, it suffices to fix $G=G_n$ for some $n\geq 2$ and prove that there exists a compact set $E\subset[0,1]^m$ and a Lipschitz map $f:E\rightarrow\XX$ such that $f(E)\supset G$. The Lipschitz constant of $f$ will depend only on $\mu(\XX)$, $m$, $s$, $n$, and $\diam G$.

\emph{The set $G$ is compact.} Since $\XX$ is complete, it suffices to prove that $G$ is closed and totally bounded. To show that $G$ is closed, suppose that $x_1,x_2,\dots\in G$ and $\lim_{k\rightarrow\infty} x_k=x$ for some $x\in \XX$ and let $0<r<1/n$. Since $B(x_k,r-|x-x_k|)\subset B(x,r)\subset B(x_k,r+|x-x_k|)$ and $0<r-|x-x_k|\leq r+|x-x_k|<1/n$ when $k$ is sufficiently large, we have \begin{equation*}\frac{1}{n}r^s = \lim_{k\rightarrow\infty} \frac{1}{n}(r-|x-x_k|)^s\leq \liminf_{k\rightarrow\infty} \mu(B(x_k,r-|x-x_k|))\leq \mu(B(x,r))\end{equation*} and \begin{equation*} \mu(B(x,r)) \leq \liminf_{k\rightarrow\infty} \mu(B(x_k,r+|x-x_k|))\leq \lim_{k\rightarrow\infty} n(r+|x-x_k|)^s=nr^s.\end{equation*} Hence $x\in G$ and the set $G$ is closed. To show that $G$ is totally bounded, let $0<r<2/n$ and let $G_r$ be a maximal set of points in $G$ such that $|x-y|>r$ for all distinct $x,y\in G$. The set $G_r$ is finite, because $$\frac{1}{n}\left(r/2\right)^s\cdot\#G_r \leq \sum_{x\in G_r} \mu\left(B\left(x,r/2\right)\right)\leq \mu(\XX)<\infty.$$  As $G\subset \bigcup_{x\in G_r} B(x,r)$ and $\#G_r<\infty$ for any $0<r<2/n$, it follows that $G$ can be covered by finitely many balls of any prescribed radius. That is, $G$ is totally bounded.

\emph{We build a tree of sets $\mathcal{T}=\bigcup_{l=0}^\infty \mathcal{T}_l$ with $\leaves(\mathcal{T})=G$.} Let $b\gg 1$ be a large number, to be specified below, with $b\rightarrow\infty$ as $s\rightarrow m$ or $n\rightarrow\infty$. At the top level, assign $\mathcal{T}_0=\{G\}$. For each $l\geq 1$, assign $\rho_l:=(1/n)b^{-l}$, $\sigma_l:=\rho_l+\rho_{l+1}+\cdots=[b/(b-1)](1/n)b^{-l}$, and $$\mathcal{T}_l:=\{G\cap B(x,\sigma_l):x\in G_{\rho_l}\}.$$ We define the parental structure as follows. As $\#\mathcal{T}_0=1$, the parent of any set in $\mathcal{T}_1$ is automatically determined. For each $l\geq 2$ and $x\in G_{\rho_l}$, choose any $x^\uparrow\in G_{\rho_{l-1}}$ such that $|x-x^\uparrow|\leq \rho_{l-1}$; then assign $G\cap B(x,\sigma_l)^\uparrow = G\cap B(x^\uparrow,\sigma_{l-1})$. Clearly, every set in $\mathcal{T}_l$ is contained in its parent, since $\rho_{l-1}+\sigma_l=\sigma_{l-1}$. Thus, $\mathcal{T}=\bigcup_{l=0}^\infty \mathcal{T}_l$ is a tree of sets in the sense of Definition \ref{def:tree}. Observe that $\bigcup\mathcal{T}_l=G$ for every $l\geq 0$. It follows that $\leaves(\mathcal{T})=\bigcap_{l=0}^\infty \bigcup\mathcal{T}_l=G$.

\emph{Rectifiability.} For each $l\geq 0$, let $N_l=\max_{F\in\mathcal{T}_l}\#\Child(F)$ and $D_l=\max_{F\in\mathcal{T}_l}\diam F$. Of course, $N_0=\#\mathcal{T}_1=\#G_{\rho_1}\lesssim_{s,n,b} \mu(\XX)$, $D_0=\diam G$, and $D_l\leq 2\sigma_l$ for all $l\geq 1$. To bound $N_l$ for $l\geq 1$, let $F=G\cap B(z,\sigma_l)\in \mathcal{T}_l$ for some $z\in G_{\rho_l}$. Using the pairwise disjointness of $\{B(x,\frac{1}{2}\rho_{l+1}):x\in G_{\rho_{l+1}}\}$ and the definition of $G$, we have \begin{equation*} \frac{1}{n}\left(\frac{1}{2n} b^{-(l+1)}\right)^s \#\Child(F) \leq \sum_{\Child(F)}\mu(B(x,\tfrac{1}{2}\rho_{l+1}))\leq \mu(B(z,\sigma_l))\leq n\left(\frac{b}{(b-1)n}b^{-l}\right)^s.\end{equation*} Hence \begin{equation*}N_l \leq n^2\left(\frac{2b^2}{b-1}\right)^s=:P_l\quad\text{for all }l\geq 1.\end{equation*} Note that, since $n> 1$ and $b\geq 2$ (at the end of the day, much bigger than 2), $$\lceil P_l^{1/m}\rceil \leq P_l^{1/m}+1=n^{2/m}\left(\frac{2b^2}{b-1}\right)^{s/m}+1<n^{2/m}(4b)^{s/m}+1<n^{2/m}(5b)^{s/m}=:C.$$ Thus, recalling that $s<m$, we see that \begin{align*}\sum_{l=1}^\infty C^l\sigma_l \leq \frac{b}{n(b-1)}\sum_{l=1}^\infty \left(\frac{n^{2/m}(5b)^{s/m}}{b}\right)^l<\infty \end{align*} provided that $b$ is large enough depending only on $m$, $s$, and $n$. Therefore, \begin{align*}
S:=\sum_{l=0}^\infty \left(\prod_{j=0}^l \lceil N_j^{1/m}\rceil\right)D_j < \lceil N_0^{1/m}\rceil\left(D_0+ 2\sum_{l=1}^\infty C^l\sigma_l\right)<\infty.\end{align*} By Theorem \ref{t:pack}, it follows that $G=\leaves(\mathcal{T})\subset f(E)$ for some compact set $E\subset[0,1]^m$ and some $S$-Lipschitz map $f:E\rightarrow\XX$. Reviewing dependencies, we see that the Lipschitz constant $S\lesssim_{m,s,n}(1+\mu(\XX)^{1/m})(1+\diam G)$.
\end{proof}

In a metric space $\XX$, a nonempty set $F\subset\XX$ is said to be \emph{$s$-homogeneous} if there exists $C>1$ such that for every bounded set $A\subset F$ and for every $\delta\in(0,1)$, there exist $C\delta^{-s}$ or fewer sets $A_1,\dots,A_n\subset F$ with $\diam A_i\leq \delta\diam A$ for all $i$ such that $A\subset A_1\cup\cdots\cup A_n$. That is, bounded subsets of $F$ can be covered by a controlled number of uniformly smaller sets. The \emph{Assouad dimension} of $F$ (see e.g.~\cite{Luukk} or \cite{Fraser-book}) can be defined as \begin{equation} \dim_A F:=\inf\{s\geq 0:\text{$F$ is $s$-homogeneous}\}.\end{equation} It is easy to see that $\dim_H F\leq \dim_A F=\dim_A{\overline{F}}$ for all $F\subset\XX$. Also, if $\dim_A \XX<\infty$, then every bounded set in $\XX$ is totally bounded and the space $\XX$ is separable. In particular, on any metric space $\XX$ with $0<\dim_A \XX<\infty$, we can find a countable dense subset $F$ of $\XX$ and this set has $\dim_H F=0< \dim_A\XX=\dim_A F$.

As a second application of the square packing construction, we extend the Lipschitz case of Badger and Vellis' \cite[Theorem 3.2]{BV} from $\RR^d$ to complete metric spaces.

\begin{theorem}\label{t:null} Let $\XX$ be a complete metric space. If $F\subset\XX$ is nonempty, $m\geq 1$ is an integer, and $\dim_A F<m$, then there exists a closed set $E\subset\RR^m$ and a Lipschitz map $f:E\rightarrow\XX$ such that $f(E)\supset F$. When $F$ is bounded, we may take $E\subset[0,1]^m$.\end{theorem}

\begin{proof} The proof for bounded sets is even easier than the proof of Theorem \ref{t:mm}, as the definition of Assouad dimension is perfectly suited to building uniform trees. The proof for unbounded sets will follow from estimates on the Lipschitz constant in the bounded case and effective use of the Attouch-Wets topology (see e.g.~\cite{Beer} or \cite{lsa}).

Suppose that $F\subset\XX$ is nonempty, bounded, and $\dim_A F<m$ for some integer $m\geq 1$. Note for later that the closure $\overline{F}$ of $F$ is compact, since $\dim_A \overline{F}=\dim_A F<\infty$ and $\overline{F}$ is bounded imply that $\overline{F}$ is totally bounded and $\XX$ is complete implies that $\overline{F}$ is complete. Now, since $\dim_A F<m$, there exists $s<m$ and $C>1$ such that $F$ is $s$-homogeneous with associated constant $C>1$. Let $b\gg C$ be a large number. We aim to build a tree a sets $\mathcal{T}=\bigcup_{l=0}^\infty \mathcal{T}_l$ with $\leaves(\mathcal{T})=F$. Assign $\mathcal{T}_0:=\{F\}$. For the induction step, suppose that we have defined $\mathcal{T}_l$ for some $l\geq 0$ so that $\diam Q\leq b^{-l}\diam F$ for all $Q\in\mathcal{T}_l$. To define $\mathcal{T}_{l+1}$, it suffices to define the set of children for each $Q\in\mathcal{T}_l$. Given $Q\in\mathcal{T}_l$, use $s$-homogeneity of $F$ with $\delta=b^{-1}$ to find $C b^s$ or fewer sets $A_1,\dots, A_N\subset F$ such that $Q\subset A_1\cup\dots\cup A_N$ and $\diam A_i\leq b^{-1}\diam Q\leq b^{-(l+1)}\diam F$ for all $i$; then assign $\Child(Q):=\{Q\cap A_1,\dots,Q\cap A_N\}$. This completes the definition of $\mathcal{T}$. For every level $l\geq 1$, we have $F\supset \bigcup\mathcal{T}_l\supset \bigcup\mathcal{T}_{l-1}\supset\cdots \supset \bigcup\mathcal{T}_0 = F$. Thus, $\leaves(\mathcal{T})=\bigcap_{l=0}^\infty \bigcup\mathcal{T}_l=F$, as desired. From construction, we see that for every $l\geq 0$, $$N_l:=\sup_{Q\in\mathcal{T}_l}\#\Child(Q)\leq Cb^{s}\quad\text{and}\quad D_l:=\sup_{Q\in\mathcal{T}_l}\diam Q\leq b^{-l}\diam F.$$ Choose $t\in(s,m)$, for concreteness say $t:=(s+m)/2$. Then, as long as $b\gg C$, we can bound $\lceil N_l^{1/m}\rceil \leq C^{1/m}b^{s/m}+1\leq (C^{1/m}+1)b^{s/m}\leq b^{t/m}$. Hence \begin{equation} S:=\sum_{j=0}^\infty \left(\prod_{i=0}^j \lceil N^{1/m}_i\rceil\right) D_j \leq b^{t/m}\diam F\sum_{j=0}^\infty b^{((t/m)-1)j}\lesssim_{m,s,C}\diam F<\infty.\end{equation} Therefore, by Theorem \ref{t:pack}, we can find a compact set $E\subset[0,1]^m$ and a $L$-Lipschitz map $f:E\rightarrow\XX$ such that $f(E)\supset \leaves(\mathcal{T})=F$, where $L\lesssim_{m,s,C} \diam F$.

Suppose that $F\subset\XX$ is unbounded and $\dim_A F<m$. Then we can find $s<m$ such that $F$ is $s$-homogeneous with associated constant $C>1$. Choose any base point $x_0\in F$. Then $F_n:=F\cap B(x_0,n)$ is also $s$-homogeneous with associated constant $C$ for each $n\geq 1$. Since $\diam F_n\leq 2n$ for each $n\geq 1$, after rescaling the domain of the maps from the bounded case, we can find a constant $L=L(m,s,C)$, compact sets $E_n\subset[0,n]^m$, and $L$-Lipschitz maps $g_n:E_n\rightarrow\XX$ such that $g_n(E_n)\supset F_n$ for each $n\geq 1$. Without loss of generality, replacing each set $E_n$ by $E_n\cap g_n^{-1}(\overline F_n)$, we may assume that $$g_n(E_n)=\overline{F_n}\subset\overline{F}\quad\text{for each }n\geq 1.$$ At this stage, it is theoretically possible that the domains of $g_n$ escape to infinity as $n\rightarrow\infty$. To correct for this, first choose points $w_n\in E_n$ such that $g_n(w_n)=x_0$; then define sets $\tilde E_n:=E_n-w_n\subset[-n,2n]^m$ and functions $\tilde g_n:\tilde E_n\rightarrow \XX$ by setting $\tilde g_n(x):=g_n(x+w_n)$. For all $n$, the function $\tilde g_n$ is $L$-Lipschitz, the image $\tilde g_n(\tilde E_n)=\overline{F_n}\subset\overline{F}$, and we have $0\in \tilde E_n$ and $\tilde g_n(0)=x_0$. In particular, neither the domains $\tilde E_n$ nor the graphs of $\tilde g_n$ escape to infinity as $n\rightarrow\infty$. Because $\RR^m$ and $\YY=\RR^m\times \overline{F}$ are proper metric spaces (i.e.~closed balls are compact), we can use the Attouch-Wets analogue of the Blaschke selection theorem\footnote{see \cite[Theorem 2.5]{lsa} or combine Theorems 3.1.4, 3.1.7, 5.1.10, 5.2.10, and 5.2.12 in \cite{Beer}} and mimic the argument from the proof of Theorem \ref{t:pack} above to produce a closed set $\tilde E\subset\RR^m$ and an $L$-Lipschitz map $f:\tilde E\rightarrow\overline{F}$ such that $f(0)=x_0$. Moreover, because $\overline{F}\supset \tilde g_n(\tilde E_n)\supset F_n\supset F_k$ for all $n\geq k$, it easily follows that $f(\tilde E)=\overline{F}\supset F$.
\end{proof}

In view of Remark \ref{r:holder}, one can also extend the H\"older case of \cite[Theorem 3.2]{BV} from $\RR^d$ to complete metric spaces.

\begin{theorem}\label{t:null-holder} Let $\XX$ be a complete metric space. If $F\subset\XX$ is nonempty, $m\geq 1$ is an integer, and $\dim_A F<sm$ for some $s>0$, then there exist a closed set $E\subset\RR^m$, a continuous map $f:E\rightarrow\XX$, and a constant $H$ such that $f(E)\supset F$ and $|f(x)-f(y)|\leq H|x-y|^{1/s}$ for all $x,y\in E$. When $F$ is bounded, we may take $E\subset[0,1]^m$.\end{theorem}

Let $\overline{\dim}_M\, \XX$ denote the upper Minkowski dimension of a metric space $\XX$ (see e.g.~\cite{Mattila} or \cite{bishop-peres}), which satisfies $\dim_H \XX\leq \dim_P \XX\leq \overline{\dim}_M \XX\leq \dim_A \XX$. Balka and Keleti recently announced the following significant extension of Theorems \ref{t:null} and \ref{t:null-holder}.

\begin{theorem}[{\cite[Theorem 4.3]{balka-keleti}}] \label{t:bk} Suppose that $\MM$ and $\XX$ are compact metric spaces. If $\overline{\dim}_M\, \XX<s \dim_H \MM$ for some $s>0$, then there exist a compact set $E\subset\MM$ and a $(1/s)$-H\"older surjection $f:E\rightarrow \XX$.\end{theorem}

The proof of Theorem \ref{t:bk} relies in part on a deep theorem of Mendel and Naor  on the existence of ultrametric subsets \cite{mendel-naor} and a theorem of Keleti, M\'ath\'e, and Zindulka on the existence of H\"older surjections from ultrametric spaces onto cubes \cite{KMZ}.

\begin{example}[H\"older parameterizations vs.~H\"older fragments] Badger and Vellis \cite{BV-IFS} proved that there exist connected self-affine Bedford-McMullen carpets $F\subset\RR^2$ such that any $(1/t)$-H\"older surjection $g:[0,1]\rightarrow F$ necessarily has $t>2$. In contrast, since the upper Minkowski dimension of the carpet $F$ is less than 2, Theorem \ref{t:bk} ensures that there exists a compact set $E\subset[0,1]$ and a $(1/s)$-H\"older surjection $f:E\rightarrow F$ with $s<2$.
\end{example}

\section{Metric cubes, dimension of measures, and other prerequisites}\label{s:background}

In this section, we collect essential tools from metric geometry and geometric measure theory that we need for the proof of Theorem \ref{t:main}.

\subsection{Generalized b-adic cubes} \label{ss:cubes}

We will use the streamlined construction of metric cubes by K\"aenm\"aki, Rajala, and Suomala \cite{KRS-cubes}. The original application (see \cite[Theorem 4.1]{KRS-cubes}) was to construct doubling measures of arbitrarily small upper packing dimension.  Recall that a set $B$ in a metric space $\XX$ is \emph{totally bounded} if $B$ can be covered by finitely many balls of radius $r$ for every $r>0$. We let $B(x,r)$ and $U(x,r)$ denote the closed and open balls in $\XX$ with center $x$ and radius $r$, respectively. We emphasize that the following theorem does not require $\XX$ to be complete. Note that we have replaced the parameter $r\in(0,1/3]$ in the original statement of the theorem with $b=r^{-1}\in[3,\infty)$.

\begin{theorem}[{\cite[Theorem 2.1]{KRS-cubes}}] \label{KRS-cubes} Let $\XX$ be a metric space in which every ball is totally bounded and choose an ``origin'' $o\in \XX$ and scaling factor $b\in[3,\infty)$. For each $k\in\ZZ$, there exists a set $\Delta_k$ of nonempty bounded Borel sets (``cubes'') and a set of points $\{x_Q:Q\in\Delta_k\}$ (``centers'') with the following properties. \begin{enumerate}
\item partitioning: $\XX=\bigcup_{Q\in\Delta_k} Q$ for all $k\in\ZZ$;
\item nesting: $Q\cap R=\emptyset$ or $R\subset Q$ for all $Q\in\Delta_k$ and $R\in \Delta_m$ when $m\geq k$;
\item roundness: $U(x_Q,c_b b^{-k})\subset Q\subset B(x_Q,C_b b^{-k})$ for all $Q\in\Delta_k$, where $$c_b=\frac{1}{2} - \frac{1}{b-1}\quad\text{and}\quad C_b= \frac{b}{b-1};$$
\item origin: for every $k\in\ZZ$, there is $Q\in\Delta_k$ such that $x_Q=o$; and,
\item inheritance: for every $k\in\ZZ$ and $Q\in \Delta_k$, there exists a cube $R\in\Delta_{k+1}$ such that $R\subset Q$ and $x_R=x_Q$.
\end{enumerate}
\end{theorem}

\begin{remark}[associated notions] \label{r:KRS} We call any family $(\Delta_k)_{k\in\ZZ}$  given by Theorem \ref{KRS-cubes} a \emph{system of $b$-adic (KRS) cubes} for $\XX$ with \emph{origin} $o$. We call $\Delta_k$ the \emph{$k$-th level} or \emph{generation} of $\Delta:=\bigsqcup_{k\in\ZZ}\Delta_k$. The same point set $Q$ in $\XX$ may belong to $\Delta_k$ for several $k$. For each $Q\in\Delta_k$, we call the number $\side Q:=b^{-k}$ the \emph{side length} of $Q$, we call $U_Q:=U(x_Q,c_b b^{-k})$ the \emph{inner ball} for $Q$, and we call $B_Q:= B(x_Q, C_bb^{-k})$ the \emph{outer ball} for $Q$. We always have $\diam Q\leq \diam B_Q\leq 2C_b \side Q.$ However, in general, there is no lower bound on $\diam Q$ in terms of $\side Q$, because $\XX$ could be bounded and/or disconnected.

For each $Q\in\Delta_k$, let $\Child(Q)=\{R\in\Delta_{k+1}:R\subset Q\}$ denote the \emph{children} of $Q$ in $\Delta$; we call $R\in\Child(Q)$ a \emph{child} of $Q$ and we call $Q$ the \emph{parent} of $R$. More generally, for any $Q\in\Delta_k$ and $j\geq 0$, let $\Child^j(Q)=\{R\in\Delta_{k+j}:R\subset Q\}$ denote the set of \emph{$j$-th generation descendants} of $Q$. (Note that $\Child^0(Q)=\{Q\}$.)

For every cube $Q\in\Delta$, we let $Q^\uparrow$ denote the parent of $Q$ and let $Q^\downarrow$ denote the unique cube $R\in \Child(Q)$ such that $x_R=x_Q$. We call $Q^\downarrow$ the \emph{central child} of $Q$. In order to write down later estimates, we define the \emph{central indicator function} $\mathsf{ci}:\Delta\rightarrow \{0,1\}$ so that $\mathsf{ci}(Q)=1$ if $Q$ is the central child of its parent (i.e.~$Q=Q^{\uparrow\downarrow}$) and $\mathsf{ci}(Q)=0$ otherwise. Finally, for every cube $Q_0\in \Delta$ and chain of descendants $Q_1\in\Child(Q_0)$, \dots, $Q_j\in\Child(Q_{j-1})$ with $j\geq 1$, we define the \emph{central counting function} \begin{equation}\label{cc-def}\mathsf{cc}(Q_0,Q_j):=\#\{1\leq i\leq j: Q_{i}=Q_{i-1}^{\downarrow}\}=\#\{1\leq i\leq j:\mathsf{ci}(Q_i)=1\}.\end{equation}\end{remark}

\begin{remark}[exhaustion] \label{r:exhaustive} Inclusion of the origin in a system of $b$-adic cubes has several consequences. For each $k\in\ZZ$, let $Q^o_k$ denote the unique cube in $\Delta_k$ whose center is the origin. For any $k_0\in\ZZ$, we have $Q^o_{k_0}\subset Q^o_{k_0-1}\subset Q^o_{k_0-2}\subset\cdots$ and $\bigcup_{j=0}^\infty Q^o_{k_0-j}=\XX$, because $Q^o_{k_0-j}\supset U(o,b^{j-k_0})$. Therefore, it is possible to exhaust $\XX$ by cubes in $\Delta$. In particular, every cube $Q\in \Delta$ belongs to $Q^o_k$ for some $k\in\ZZ$. Consequently, every pair of cubes in $\Delta$ have a common ancestor. \end{remark}

\begin{remark}[the central child is relatively far away from the boundary] \label{r:central} Recall that for any nonempty sets $A$ and $B$ in $\XX$, $\gap(A,B)=\inf_{a\in A}\inf_{b\in B}|a-b|$. A key property that we will use without further comment is monotonicity: $\gap(A,B)\geq \gap(C,D)$ whenever $A\subset C$ and $B\subset D$. If the scaling factor $b>5$ and $Q\in\Delta_k$ is any cube such that $Q\neq \XX$, then \begin{equation}\begin{split} \label{b-gap} \gap(Q^{\downarrow},\XX\setminus Q)&\geq \gap(B_{Q^\downarrow},\XX\setminus U_Q)\\
&\geq \radius U_{Q}-\radius B_{Q^\downarrow} =(bc_b-C_b)b^{-(k+1)}\gtrsim_b \side Q^\downarrow.\end{split}\end{equation} Indeed, $bc_b-C_b=b(\frac{1}{2}-\frac{2}{b-1})>0$ if and only if $b>5$.
\end{remark}

\begin{example} A modified version of the usual triadic cubes in $\RR^d$ enjoys all of the essential properties of KRS cubes. Let $o=(0,\dots,0)$ denote the origin. We initially declare that each cube of the form $Q_k^o=(-\frac{1}{2}\cdot 3^k,\frac12\cdot 3^k]^d$ with $k\in\ZZ$ is an \emph{origin-based triadic cube}. Further, we declare that any half-open cube that appears after trisecting an origin-based triadic cube into $3^d$ equal size subcubes is also an origin-based triadic cube. Let $\Delta_k$ denote all origin-based triadic cubes of side length $3^{-k}$. The center $x_Q$ of $Q\in\Delta=\bigcup_{k\in\ZZ}\Delta_k$ is the geometric center of the cube. The family $(\Delta_k)_{k\in\ZZ}$ of origin-based triadic cubes satisfy properties (1)--(5) of Theorem \ref{KRS-cubes} with scaling factor $b=3$ and with constants $c_b$ and $C_b$ replaced by $1/2$ and $\sqrt{n}/2$, respectively. Of course, $\gap(Q^\downarrow, \RR^d\setminus Q)\geq \side Q^\downarrow$ for all $Q$. One way in which origin-based triadic cubes are superior to the usual variant is that each nonempty bounded set in $\RR^d$ is contained in an origin-based triadic cube. Another improvement is that any two origin-based triadic cubes have a common ancestor.
\end{example}

\begin{lemma}[counting cubes I]\label{l:AR-count}
Let $\XX$ be an Ahlfors $q$-regular metric space. If $(\Delta_k)_{k\in\ZZ}$ is a system of $b$-adic cubes for $\XX$ with $b\geq 5$, then \begin{equation}\label{e:AR-count-X} b^{jq}\lesssim_\XX \#\Child^j(Q) \lesssim_\XX b^{jq}\quad\text{for all $Q\in\Delta$ with $Q\neq\XX$ and $j\geq 1$.}\end{equation}
\end{lemma}

\begin{proof} Since $\XX$ is $q$-regular, there exists a measure $\nu$ and constants $C,D>0$ such that $Cr^q\leq \nu(U(x,r))\leq \nu(B(x,r))\leq CD r^q$ for all $x\in\XX$ and for all $0<r<\diam X$. Let $Q\in\Delta_k$ with $Q\neq \XX$. Then we may use the lower bound on $\nu(U_Q)$ and on $\nu(U_R)$ for any descendant $R$ of $Q$. The upper bounds on $\nu(B_Q)$ and $\nu(B_R)$ are always valid. Let $j\geq 1$. Using $U_Q\subset Q\subset \bigcup_{R\in\Child^j(Q)}B_R$ and $\bigcup_{R\in\Child^j(Q)}U_R\subset Q\subset B_Q$, as well as the pairwise disjointness of $\{U_R:R\in\Child^j(Q)\}$, we have $$Cc_b^qb^{-kq} \leq \#\Child^j(Q)CDC_b^qb^{-(k+j)q}\text{\ \ and\ \ }\#\Child^j(Q)Cc_b^qb^{-(k+j)q}\leq CDC_b^qb^{-kq}.$$ Rearranging yields $[D(C_b/c_b)^q]^{-1} b^{jq}\leq \#\Child^j(Q) \leq D(C_b/c_b)^q b^{jq}$. Noting that $C_b/c_b=2b/(b-3)$ yields \begin{equation}\label{e:AR-count} \left[(\tfrac{2b}{b-3})^q D\right]^{-1} b^{jq}\leq \#\Child^j(Q) \leq \left[(\tfrac{2b}{b-3})^q D\right] b^{jq}.\end{equation} When $b\geq 5$, we have $C_b/c_b\leq C_5/c_5=5$. Thus, letting $\nu$ range over all possible Ahlfors $q$-regular measures on $\XX$, it follows that we may replace $D(C_b/c_b)^q$ in \eqref{e:AR-count} with a constant depending on $\XX$ (including $q$), but not on a choice of $\nu$ nor on the choice of a particular system of cubes.
\end{proof}

To construct doubling measures using $b$-adic cubes (see \S\ref{s:bernoulli}), it will be convenient to have the following variant of Lemma \ref{l:AR-count}.

\begin{lemma}[counting cubes II]\label{l:2AR-count} Let $\XX$ be an Ahlfors $q$-regular metric space and let  $(\Delta_k)_{k\in\ZZ}$ be a system of $b$-adic cubes for $\XX$ with $b\geq 47$. For all $Q\in\Delta_k$, define \begin{align}
\Inner(Q)&:=\{R\in\Child(Q):R\cap U(x_Q,\tfrac12c_b b^{-k})\neq\emptyset\} \label{inner-def},\\
\Outer(Q)&:=\Child(Q)\setminus\Inner(Q)=\{R\in\Child(Q):R\cap U(x_Q,\tfrac12c_b b^{-k})=\emptyset\} \label{outer-def}.\end{align} For all $Q\in\Delta$ with $Q\neq\XX$, we have \begin{equation}\label{in-out-count} b^q\lesssim_\XX\#\Inner(Q)\lesssim_\XX b^q\quad\text{and}\quad \#\Outer(Q)\lesssim_\XX b^q.\end{equation} If $Q\in\Delta$, $Q\neq\XX$, $R\in\Child(Q)$, and $\gap(R,\XX\setminus Q)\leq 9\side R$, then $R\in\Outer(Q)$.
\end{lemma}

\begin{proof} Let $Q\in\Delta_k$ with $Q\neq \XX$. We start with the final claim. For all $R\in\Child(Q)$, $$\diam R\leq \frac{2b}{b-1}\side R<2.1\side R\quad\text{and}\quad\radius U_Q=\frac{b-3}{2b-2} b\side R> 22.4\side R,$$ since $b\geq 47$. Also, $\gap(\frac{1}{2}U_Q,\XX\setminus Q)\geq \gap (\frac{1}{2} U_Q,\XX\setminus U_Q)\geq \frac{1}{2}\radius U_Q$. Hence \begin{equation*}\begin{split}\gap(\tfrac{1}{2}U_Q,R)&\geq \gap(\tfrac{1}{2}U_Q,\XX\setminus Q)-\diam R-\gap(R,\XX\setminus Q)\\ &> 11.2\side R-2.1\side R - \gap(R,\XX\setminus Q)>0\end{split}\end{equation*} and $R\in\Outer(Q)$ if $\gap(R,\XX\setminus Q)\leq 9 \side R$.

Let $\nu$, $C$, and $D$ be as in the proof of Lemma \ref{l:AR-count}. By definition of $\Inner(Q)$, we have $\frac{1}{2} U_Q\subset \bigcup_{R\in\Inner(Q)} B_R$. Hence using the Ahlfors regularity inequalities for $\nu$, $$C(\tfrac{1}{2}c_b)^qb^{-kq} \leq \#\Inner(Q)CDC_b^qb^{-(k+1)q}.$$ Since $2C_b/c_b=4b/(b-3)< 4.3$ when $b\geq 47$, \begin{equation} \#\Inner(Q) \geq (4.3D)^{-1}b^q.\end{equation} Letting $\nu$ range over all Ahlfors $q$-regular measures on $\XX$ yields the lower bound on $\#\Inner(Q)$ in \eqref{in-out-count}. Lemma \ref{l:AR-count} gives the upper bounds.
\end{proof}

\begin{remark} The constants $47$ and $9$ in the expressions $b\geq 47$ and $\gap(R,\XX\setminus Q)\leq 9\side R$ are a convenient choice for the proof of Lemma \ref{qb-doubling}. In general, there is no analogue of the lower bound on $\#\Inner(Q)$ in \eqref{in-out-count} for $\#\Outer(Q)$ and it is possible that $\Outer(Q)=\emptyset$. \end{remark}

\begin{corollary} \label{c:good-count} Let $\XX$ be an Ahlfors $q$-regular metric space with $\diam\XX\geq 2.1$ and let $(\Delta_k)_{k\in\ZZ}$ be a system of $b$-adic cubes for $\XX$ with $b\geq 47$. If $0<s<q$ and $b$ is sufficiently large depending only on $\XX$ (including $q$) and $s$, then \begin{equation}\label{good-count} \#\Inner(Q)\geq b^s \quad\text{and}\quad \#\Child(Q)\leq b^{q+1}\end{equation} for all $Q\in\Delta_+=\bigcup_{k=0}^\infty \Delta_k$.
\end{corollary}

\begin{proof} For any $Q\in\Delta_+$, we have $\diam Q\leq 2C_b/c_b<2.1\leq \diam\XX$ and $Q\neq\XX$, since $b\geq 47$. By Lemma \ref{l:AR-count} and \ref{l:2AR-count}, we can find $C_\XX>1$ such that $\#\Inner(Q)\geq C_\XX^{-1} b^q$ and $\#\Child(Q)\leq C_\XX b^q$. Then, if $b$ is sufficiently large, we have $C_\XX^{-1}\geq b^{s-q}$ and $C_\XX\leq b$.\end{proof}

\subsection{Dimension of measures}\label{ss:dimension} Let $\XX$ be a metric space. To set conventions, we define the \emph{$s$-dimensional Hausdorff measure} $\Haus^s(E)$ and \emph{$s$-dimensional packing measure} $\Pack^s(E)$ for all $s\in[0,\infty)$ and for all nonempty $E\subset\XX$ by \begin{equation}\Haus^s(E):=\lim_{\delta\downarrow 0} \inf\left\{\textstyle\sum_{i=1}^\infty (\diam E_i)^s : E\subset\textstyle\bigcup_{i=1}^\infty E_i,\,\forall_{i\geq 1}\diam E_i\leq \delta\right\}\end{equation} and \begin{equation} \begin{split} \Pack^s(E)&:=\inf\left\{\textstyle\sum_{i=1}^\infty P^s(E_i):E\subset\textstyle\bigcup_{i=1}^\infty E_i\right\},\text{ where}\\
P^s(E)&:=\lim_{\delta\downarrow 0}\sup\left\{\textstyle\sum_{i=1}^\infty (2r_i)^s: \forall_{i\geq 1} x_i\in E,  r_i\in(0,\delta/2],\,\right. \\
&\qquad\qquad\qquad\qquad\qquad\  \left. \forall_{i\neq j} B(x_i,r_i)\cap B(x_j,r_j)=\emptyset\,\right\}.\end{split}\end{equation} Then $\dim_H E:=\inf\{s\geq 0:\Haus^s(E)=0\}$ and $\dim_P E:=\inf\{s\geq 0:\Pack^s(E)=0\}$, using the convention that $\inf\emptyset=\infty$. Because a general metric space is not necessarily uniformly perfect\footnote{A metric space $\XX$ is $c$-uniformly perfect if $\diam B(x,r)\geq cr$ for all $x\in \XX$ and $r>0$ with $B(x,r)\neq\XX$. For example, connected metric spaces are $1$-uniformly perfect.}, it is important to adopt the ``radial'' definition of the packing measure (see \cite{Cutler-density-theorem}) and we have done so.

The \emph{lower} and \emph{upper Hausdorff dimensions} of a Borel measure $\mu$ on a separable metric space $\XX$ are given respectively by \begin{align} \underline{\dim}_H\, \mu &:=\inf\{\dim_H E:\mu(E)>0,\;E\text{ Borel}\},\\ \overline{\dim}_H\, \mu &:=\inf\{\dim_H E:\mu(\XX\setminus E)=0,\;E\text{ Borel}\}.\end{align} The \emph{support} of $\mu$, denoted by $\spt\mu$, is the smallest closed set $F$ such that $\mu(\XX\setminus F)=0$. Naturally, $\underline{\dim}_H\,\mu\leq \overline{\dim}_H\,\mu\leq \dim_H\spt\mu$ for all $\mu$, but both inequalities can be strict. That is to say, the dimension of a measure is a distinct notion from the dimension of its support. The \emph{lower} and \emph{upper packing dimensions} of $\mu$ are defined analogously by substituting the packing dimension of $E$ for the Hausdorff dimension of $E$. In general, \begin{equation}\underline{\dim}_H\,\mu\leq \underline{\dim}_P\,\mu;\quad \overline{\dim}_H\,\mu\leq \overline{\dim}_P\,\mu;\quad\text{and}\quad \underline{\dim}_P\,\mu\leq \overline{\dim}_P\,\mu.\end{equation} In the rare situation that $\underline{\dim}_H\,\mu=\overline{\dim}_H\,\mu=\underline{\dim}_P\,\mu=\overline{\dim}_P\,\mu=s$, we may say that $\mu$ has \emph{exact} dimension $s$. The following well-known formulas for dimensions of measures in Euclidean space $\RR^d$ (see e.g.~\cite{Falconer-techniques}) continue to persist in the metric setting:

\begin{theorem}[{Tamashiro \cite[Theorem 1.8]{tamashiro}}] \label{metric-local-dim} Let $\XX$ be a separable metric space. If $\mu$ is a finite Borel measure on $\XX$, then \begin{align}
  \label{dim1} \underline{\dim}_H\,\mu &= \mu\dash\essinf_{x\in\XX} \left(\liminf_{r\downarrow 0} \frac{\log \mu(B(x,r))}{\log r}\right),\\
  \label{dim2} \overline{\dim}_H\,\mu &= \mu\dash\esssup_{x\in\XX} \left(\liminf_{r\downarrow 0} \frac{\log \mu(B(x,r))}{\log r}\right),\\
  \label{dim3} \underline{\dim}_P\,\mu &= \mu\dash\essinf_{x\in\XX} \left(\limsup_{r\downarrow 0} \frac{\log \mu(B(x,r))}{\log r}\right),\\
  \label{dim4} \overline{\dim}_P\,\mu &= \mu\dash\esssup_{x\in\XX} \left(\limsup_{r\downarrow 0} \frac{\log \mu(B(x,r))}{\log r}\right).
\end{align}\end{theorem}

The quantities $\liminf_{r\downarrow 0} \log\mu(B(x,r))/\log(r)$ and $\limsup_{r\downarrow 0}\log\mu(B(x,r)/\log(r)$ are called the \emph{lower} and \emph{upper local dimensions} of $\mu$ at $x$, respectively.
On doubling metric spaces, one may replace balls $B(x,r)$ in \eqref{dim1}--\eqref{dim4} with $b$-adic cubes.

\begin{lemma}[see {\cite[Proposition 3.1]{KRS-multifractal}}]  Let $\XX$ be a doubling metric space and let $(\Delta_k)_{k\in\ZZ}$ be a system of $b$-adic cubes for $\XX$. Let $Q_k(x)$ denote the cube in $\Delta_k$ containing $x\in\XX$. If $\mu$ is a Borel measure on $\XX$ that is finite on bounded sets, then at $\mu$-a.e.~$x\in\XX$, \begin{align}
  \label{dim5} \liminf_{r\downarrow 0} \frac{\log \mu(B(x,r))}{\log r}&=\liminf_{k\rightarrow\infty} \frac{\log \mu(Q_k(x))}{\log b^{-k}}, \\
  \label{dim6} \limsup_{r\downarrow 0} \frac{\log \mu(B(x,r))}{\log r}&=\limsup_{k\rightarrow\infty} \frac{\log \mu(Q_k(x))}{\log b^{-k}}.
\end{align}
\end{lemma}

\subsection{Measures with prescribed values}

In the remainder of the paper, we will study the properties of certain measures on metric spaces $\XX$ whose balls are totally bounded, formally built by \begin{enumerate}
\item choosing a system of $b$-adic cubes $(\Delta_k)_{k\in\ZZ}$ on $\XX$;
\item specifying a function $w:\{\overline{Q}:Q\in\Delta_+\}\rightarrow [0,\infty)$ on the closure of cubes in $\Delta_+=\bigcup_{k=0}^\infty\Delta_k$ such that $w(\overline{Q})= \sum_{R\in\Child(Q)}w(\overline{R})$ for all $Q\in\Delta_+$ and extending the definition by assigning $w(\emptyset)=0$; and,
\item defining $\mu_w(E):=\inf\{\sum_{i=1}^\infty w(\overline{Q_i}): E\subset \bigcup_{i=1}^\infty \overline{Q_i}\text{ for some }Q_1, Q_2,\cdots \in \Delta_+\cup\{\emptyset\}\}$ for all $E\subset\XX$.
\end{enumerate} Because the weight $w$ is additive over children, we may alternatively write $$\mu_w(E)=\lim_{\delta\rightarrow 0} \inf\left\{\sum_{i=1}^\infty w(\overline{Q_i}): E\subset \bigcup_{i=1}^\infty\overline{Q}_i \text{ for some }Q_1, Q_2,\cdots \in \Delta_+\cup\{\emptyset\},\,\diam Q_i\leq \delta\right\}$$ It easily follows that $\mu_w$ is a metric outer measure and Borel sets are $\mu_w$ measurable. Further, because any ball can be covered by a finite number of cubes of side length 1 and the outer measure is defined using outer approximation by closed sets, it follows that $\mu_w$ is a Radon measure on $\XX$, i.e.~a locally finite Borel regular outer measure on $\XX$. For details, see e.g.~\cite{Rogers}. However, it is an unpleasant reality that the measures $\mu_w(Q)$ and $\mu_w(\overline{Q})$ and the weight $w(\overline{Q})$ do not need to agree on cubes $Q\in\Delta_+$. There are two issues.

\begin{example}Let $\XX=\RR$ and let $(\Delta_k)_{k\in\ZZ}$ be the system of left-open triadic intervals. Define a weight $w$ so that $w([0,3^{-k}])=1$ for all $k\geq 0$, and $w(\overline{I})=0$ for all other triadic intervals of length at most 1. Then $\mu_w([0,1])=0$ even though $w([0,1])=1$. To see this, simply note that $$[0,1]\subset[-1,0]\cup\bigcup_{k=1}^\infty [3^{-k},2\cdot 3^{-k}]\cup [2\cdot 3^{-k}, 3\cdot 3^{-k}]$$ and the weight of each closed triadic interval on the right hand side is zero. The difficulty in this example is that $w$ is not countably subadditive.\end{example}

\begin{example}Let $\XX=\QQ$ be equipped with the subspace metric from $\RR$, which is a doubling metric space. Let $(\Delta_k)_{k\in\ZZ}$ be the system of left-open triadic intervals (in $\QQ$). Define a weight $w$ so that $w(\QQ\cap[n,n+1])=1$ for all $n\in\ZZ$ and on any triadic interval $I=L\cup M\cup R$ of side length at most 1, $w(\overline{L})=w(\overline{M})=w(\overline{R})=(1/3)w(\overline{I})$, where $L$, $M$, and $R$ are the left, middle, and right triadic children of $I$, respectively. In contrast to the previous example, the weight $w$ has the nice property that $w$ is centrally doubling insofar as $w(\overline{I^\downarrow})=w(\overline{M})\gtrsim w(\overline{I})$ for all $I\in\Delta_+$. Nevertheless, $\mu_w$ is the zero measure. Indeed, $\mu_w$ is nothing other than the restriction of Lebesgue outer measure on $\RR$ to the power set of $\QQ$. Since $\QQ$ is countable, $\mu_w(\QQ)=0$. The difficulty in this example is that $\QQ$ is not complete.\end{example}

The following criterion is sufficient to ensure that $\mu_w$ takes prescribed values. It gives one possible solution to the technical issue described in \cite[Remark 5.1(1)]{KRS-cubes}.

\begin{lemma}[extension criterion]\label{extension} Assume that $\XX$ is a proper metric space (i.e.~every closed ball in $\XX$ is compact) and $b>5$. If there exists a constant $0<p\leq 1$ such that $w(\overline{Q^\downarrow})\geq p w(\overline{Q})$ and $w(\overline{Q})= \sum_{R\in\Child(Q)}w(\overline{R})$ for all $Q\in\Delta_+$, then $\mu_w(\partial Q)=0$ and $\mu_w(\interior{Q})=\mu_w(Q)=\mu_w(\overline{Q})=w(\overline{Q})$ for all $Q\in\Delta_+$.\end{lemma}

\begin{proof} Let $Q\in\Delta_+$. If it happens that $Q=\XX$, then $\mu_w(\partial Q)=\mu_w(\emptyset)=0$. Suppose that $Q\neq\XX$. Since $b>5$, there exists a constant $\delta=\delta(b)>0$ such that $\gap(P^\downarrow,\XX\setminus P)\geq \delta \side P^\downarrow >0$ for all $P\in\Delta$ such that $P\neq\XX$ by Remark \ref{r:central}. Hence $\partial Q$ is covered by $$\{\overline{R}:R\in\Child^j(Q)\text{ and } \mathsf{cc}(Q,R)=0\},$$ where $\mathsf{cc}(Q,R)$ is given by \eqref{cc-def}. Thus, by the hypothesis, $\mu_w(\partial Q)\leq (1-p)^j w(\overline{Q})$ for all $j\geq 1$. It follows that $\mu_w(\partial Q)=0$ and $\mu_w(\interior{Q})=\mu_w(Q)=\mu_w(\overline{Q})\leq w(\overline{Q})$. To finish, it suffices to show that $w$ is countably subadditive, i.e.~$w(\overline{Q})\leq \sum_{n=1}^\infty w(\overline{Q}_n)$ whenever $\overline{Q}\subset\bigcup_{n=1}^\infty \overline{Q}_n$ for some sequence $Q_n\in\Delta_+\cup\{\emptyset\}$, because this implies $w(\overline{Q})\leq \mu_w(\overline{Q})$.

Suppose that $\overline{Q}\subset \bigcup_{n=1}^\infty \overline{Q}_n$ for some sequence $Q_n\in\Delta_+\cup\{\emptyset\}$. Fix $\epsilon>0$. By the same reasoning as in the previous paragraph, for each $n\geq 1$, we can choose $j=j(Q_n)$ sufficiently large so that \begin{align*}\mathsf{Aux}(Q_n):=\{R\in \Child^j(P):\, &P\in\Delta_+,\; \side P=\side Q_n,\;\\
 &\gap(P,Q_n)<\side Q_n,\; \mathsf{cc}(P,R)=0\}\end{align*} satisfies $\sum_{R\in\mathsf{Aux}(Q_n)}w(\overline{R})\leq 2^{-n}\epsilon$. (The set of all $P\in\Delta_+$ such that $\side P=\side Q_n$ and $\gap(P,Q)\leq \side Q_n$ is finite, because the set $\{U_P\}$ of associated inner balls are pairwise disjoint, have the same radius $c_b \side Q_n$, are contained in $B(x_{Q_n},(2C_b+1)\side Q_n)$, and balls in $\XX$ are totally bounded.) For each $n\geq 1$, define the open set $$U_n:=\bigcup_{x\in Q_n} U\left(x,\tfrac{1}{2}\delta b^{-j(Q_n)}\side Q_n\right).$$ If $y\in U_n\setminus \{\overline{Q}_n\}$ and $P_y\in\Delta_+\setminus\{Q_n\}$ is the cube containing $y$ with $\side P_y=\side Q_n$, then $\gap(P_y,Q_n)\leq \dist(y,Q_n)<\frac{1}{2}\delta b^{-j}\side Q_n<\side Q_n$; moreover, if $R_y\in\Child^{j(Q_n)}(P_y)$ is the descendent containing $y$, then $R_y\in\mathsf{Aux}(Q_n)$. For the latter claim, simply note that any $R\in\Child^{j(Q_n)}(P_y)$ with $\mathsf{cc}(P_y,R)\geq 1$ has $$\gap(R,Q_n)\geq \gap(R,\XX\setminus P_y)\geq \delta \side R=\delta b^{-j(Q_n)}\side Q_n>2\dist(y,Q_n),$$ whence $y\not\in R$. Everything considered, $U_n$ is an open set and $\overline{Q}_n\subset U_n\subset \overline{Q}_n\cup\bigcup\mathsf{Aux}(Q_n)$. It follows that $\{U_n:n\geq 1\}$ is an open cover of $\overline{Q}$, since $\{\overline{Q}_n: n\geq 1\}$ covers $\overline{Q}$, and  $\overline{Q}$ is compact, since $\XX$ is proper. Hence, after relabeling, we can assume that $U_1,\dots, U_k$ cover $\overline{Q}$ for some $k$. In particular, $$\mathcal{F}:=\{\overline{Q}_1,\dots,\overline{Q}_k\}\cup\bigcup_{j=1}^k \{\overline{R}:R\in\mathsf{Aux}(Q_j)\}$$ is a finite cover of $\overline{Q}$, with the total weight of auxiliary cubes $\sum_{j=1}^k\sum_{R\in\mathsf{Aux}(Q_j)}w(\overline{R})\leq \epsilon$. The subfamily $\{S\in\mathcal{F}:S\cap Q\neq\emptyset\}$ also covers $\overline{Q}$. Let $\mathcal{F}'$ be any minimal subcover of $\{S\in\mathcal{F}:S\cap Q\neq\emptyset\}$.
Because $\mathcal{F}'$ is finite, we can use finite additivity of $w$ to obtain $$w(\overline{Q})=\sum_{S\in\mathcal{F}'}w(S)\leq \sum_{S\in\mathcal{F}}w(S) \leq \epsilon+\sum_{j=1}^\infty w(\overline{Q}_n).$$ Sending $\epsilon\to 0$ confirms that $w$ is countably subadditive.
\end{proof}

\begin{remark}[how to use this practically] \label{practical} On a proper metric space $\XX$, we choose $b>5$ and fix a system  $(\Delta_k)_{k\in\ZZ}$ of $b$-adic cubes. To define a Radon measure $\mu$ on $\XX$ with prescribed values on $\Delta_+$, we may casually \begin{enumerate}
\item assign some arbitrary finite value $\mu(Q)$ for each $Q\in\Delta_0$ and
\item describe how to distribute the mass $\mu(Q)$ for each $Q\in\Delta_+$ among its children in any way such that $\mu(Q)=\sum_{R\in\Child(Q)}\mu(R)$ and $\mu(Q^\downarrow)\geq p\,\mu(Q)$ with $0<p\leq 1$ independent of $Q$.
\end{enumerate} Then the measure $\mu$ exists and is unique. Indeed, to show existence, define a weight $w$ such that $w(\overline{Q}):=\mu(Q)$ for all $Q\in\Delta_+$. Then $\mu_w$ is a Radon measure with $\mu_w(Q)=\mu(Q)$ and $\mu_w(\partial Q)=0$ for all $Q\in\Delta_+$ by Lemma \ref{extension}. We then relabel $\mu_w$ as $\mu$ and forget about the weight $w$. For uniqueness, simply note that the values of $\mu$ on $\Delta_+$ determine the values of $\mu$ on open sets, and thus, determine the values of $\mu$ on arbitrary sets, because $\mu$ is Radon.\end{remark}

\section{Estimates for quasi-Bernoulli measures}\label{s:bernoulli}

We begin by describing the measures in Theorem \ref{t:main} in a special case. The definition of the measures will rely in part on the following calculation, a simple exercise in calculus. For an introduction to the concept of entropy, we refer the reader to \cite{Cover-Thomas}.

\begin{lemma}\label{entropy-1} For every scaling factor $b>1$ and integer $N\geq 1$, the entropy function $h_{b,N}:(0,1/N]\rightarrow (0,\infty)$ given by \begin{equation} h_{b,N}(\delta):=(N-1)\delta\log_b\left(\frac{1}{\delta}\right)+(1-(N-1)\delta)\log_b\left(\frac{1}{1-(N-1)\delta}\right)\end{equation} is differentiable, monotone increasing, $h_{b,N}(0+)=0$, and $h_{b,N}(1/N)=\log_b(N)$.\end{lemma}

\begin{example}[Euclidean space] \label{ex:Euc-1} Let $1\leq m\leq d-1$ be integers and let $s\in(m-1,m)$. We define a self-similar Bernoulli-type measure $\mu$ on $\XX=\RR^d$ by specifying its values on (origin-based or standard) triadic cubes in $\RR^d$ as follows. Let $\delta>0$ be a small number, whose exact value depending on $d$ and $s$ will be given momentarily.\begin{enumerate}
\item Declare $\mu(Q)=1$ for any triadic cube of side length 1.
\item For any triadic cube $Q$ of side length at most 1, declare $\mu(R)=\delta \mu(Q)$ for all non-central children $R$ of $Q$ and  declare $\mu(Q^\downarrow)=(1-(3^d-1)\delta)\mu(Q)$.
\end{enumerate} That is to say, below scale 1, $\mu$ is defined by concentrating most of the mass in the center of a cube. The measure $\mu$ is a doubling measure (see \cite[\S8.2]{BLZ} or \cite[\S3]{KRS-cubes}  for sample details) and the Hausdorff and packing dimension of $\mu$ is exact (see e.g.~\cite{dimension-of-measures} or \cite[\S1.5]{bishop-peres}) and is given by the entropy formula \begin{equation} h_{3,3^d}(\delta)=(3^d-1)\delta\log_3\left(\frac{1}{\delta}\right)+(1-(3^d-1)\delta)\log_3\left(\frac{1}{1-(3^d-1)\delta}\right).
\end{equation} To force $\mu$ to have dimension $s$, we simply choose $\delta=\delta(d,s)$ so that $h_{3,3^d}(\delta)=s$. Because $s>m-1$, we immediately see that $\mu$ is purely $(m-1)$-unrectifiable. Since $0<s<d$, it can be shown using the law of the iterated logarithm (see e.g.~\cite[Theorem 3.1]{measure-lil} or \cite[Theorem 1.1]{ifs-density}) that at $\mu$-a.e.~$x\in\RR^d$, \begin{equation}\label{limsup-infinity} \displaystyle\liminf_{r\downarrow 0}\dfrac{\mu(B(x,r))}{r^s}=\displaystyle\liminf_{k\rightarrow\infty}\dfrac{\mu(Q_k(x))}{3^{-ks}}=0,\quad \displaystyle\limsup_{r\downarrow 0}\dfrac{\mu(B(x,r))}{r^s}=\displaystyle\limsup_{k\rightarrow\infty}\dfrac{\mu(Q_k(x))}{3^{-ks}}=\infty,\end{equation} where $Q_k(x)$ denotes the triadic cube of side length $3^{-k}$ containing $x$. This means that it is impossible to use Theorem \ref{t:mm} (or the earlier results of \cite{MM1988} or \cite{BV}) to verify that $\mu$ is $m$-rectifiable. Nevertheless, using some estimates developed later in this section (inspired by \cite{GKS}) together with either Theorem \ref{t:pack} or Corollary \ref{c:pack}, it can be shown that $\mu$ is $m$-rectifiable (see \S\ref{ss:simple}).
\end{example}

Up to technical details, the measures in Theorem \ref{t:main} on an Ahlfors regular metric space are defined analogously to the measures in Example \ref{ex:Euc-1}. All difficulties stem from imprecise and locally varying counts of metric cubes. Instead of using two weights per cube to define the mass of children, we will need three. See Figure \ref{fig:qb}.

\begin{lemma}\label{entropy-2} If $b>1$ and $L$ and $M$ are integers such that $0\leq L\leq b^y$ and $M\geq b^s$ for some $s,y>0$, then there exists a number $\alpha_0=\alpha_0(b,y,s)$ such that for all $0\leq \alpha\leq \alpha_0$, there exist unique numbers $\beta=\beta(\alpha,b,y,s,L,M)$ and $\gamma=\gamma(\alpha,b,y,s,L,M)$ such that \begin{equation}\label{abc-1} L\alpha+(M-1)\beta+\gamma=1\end{equation} and the entropy function
\begin{equation}\label{abc-2}
h_{b,L,M}(\alpha,\beta):=L\alpha\log_b(1/\alpha) +(M-1)\beta\log_b(1/\beta)+\gamma\log_b(1/\gamma)=s.\end{equation} We may always bound $L\alpha\log_b(1/\alpha)\leq \min(1,s)/e$, $L\alpha\leq \min(1,s^2)/e^2$,\begin{equation}\label{gamma-bound} \gamma\geq 1-L\alpha-\frac{s-L\alpha\log_b(1/\alpha)}{\log_b(M-1)} \geq 1-\frac{\min(1,s^2)}{e^2}-\left(1-\frac{1}{e}\right)\frac{s}{\log_b(M-1)},\end{equation} \begin{equation}\label{gamma-bound-2} \text{and}\qquad \gamma\geq \frac{1-L\alpha}{M}\geq \frac{1}{M}\left(1-\frac{\min(1,s^2)}{e^2}\right)\geq \frac{1}{M}\left(1-\frac{1}{e^2}\right).\end{equation} Moreover, if $2e^2\log_b(e^2)\leq (\tfrac{1}{2}-\tfrac{1}{e})s$, then \begin{equation}\label{beta-bound} \beta\geq \frac{s}{2(M-1)\log_b(M-1)}.\end{equation}\end{lemma}

\begin{proof} Assign $z:=\max(1,y)$ and $t:=\min(1,s)$. The function $x\mapsto x\log_b(1/x)$ is strictly increasing on $(0,1/e]$ and takes its maximum value $1/(e\ln(b))$ at $x=1/e$. Since $$b^{z}\geq z\ln(b)\geq \ln(b)\geq (t/e)e\ln(b),$$ there is a unique number $a\in(0,1/e]$ such that $b^za\log_b(1/a)=t/e$. Using the comparison $\ln(x)\leq (2/e)x^{1/2}$ for $x>0$, we see that $$a = \frac{t\ln(b)}{eb^z\ln(1/a)} \geq a^{1/2}\frac{ t\ln(b)}{2b^z}\quad\text{whence}\quad a \geq \left(\frac{t\ln(b)}{2b^z}\right)^{2}.$$ With partial foresight, we choose  \begin{equation}\label{alpha-def} \alpha_0:=\min\bigg(\Big(\frac{t\ln(b)}{2b^z}\Big)^{2}, \frac{1}{b^s}\bigg),\qquad\bigg[\text{note that } \alpha_0=\Big(\frac{t\ln(b)}{2b^z}\Big)^{2}\text{ when }s\leq y,\bigg] \end{equation}  which depends only on $b$, $y$, and $s$. Let $0\leq \alpha\leq \alpha_0$ be given. Using $\alpha\leq \alpha_0\leq \min(a,b^{-s})$ and the upper bound on $L$, we see that\begin{equation}\label{alpha-bounds} \begin{array}{c} L\alpha s\leq L\alpha\log_b(1/\alpha) \leq b^z a\log_b(1/a) = \min(1,s)/e, \vspace{.2cm}\\
\text{and}\quad L\alpha \leq t^2\ln(b)^2/4b\leq \min(1,s^2)/e^2.\end{array}\end{equation} To continue, we abbreviate $w:=1-L\alpha\in(0,1]$ and consider the function $$h(x):=(M-1)x\log_b\left(\frac{1}{x}\right)+(w-(M-1)x)\log_b\left(\frac{1}{w-(M-1)x}\right)\quad\text{on }(0,w/M].$$ Observe that $h$ is differentiable and strictly increasing on $(0,w/M]$, $h(0+)=0$, and $$h(w/M)=w \log_b(M/w)\geq w\log_b(M)\geq (1-L\alpha)s\geq s-L\alpha\log_b(1/\alpha)\geq s(1-1/e)$$ by the lower bound on $M$ and \eqref{alpha-bounds}. Thus, there exists some unique $\beta=\beta(\alpha,b,y,s,L,M)$ such that $h(\beta)=s-L\alpha\log_b(1/\alpha)$. Setting $\gamma:=w-(M-1)\beta$, we arrive at the desired conditions \eqref{abc-1} and \eqref{abc-2}.

To find lower bounds for $\beta$ and $\gamma$, set $\epsilon=(M-1)\beta$ and write \begin{equation}\begin{split}\label{entropy-expansion}
s-L\alpha\log_b(1/\alpha)&=\epsilon\log_b(M-1)
   +\underbrace{\epsilon\log_b\left(\frac{1}{\epsilon}\right)+(w-\epsilon)\log_b\left(\frac{1}{w-\epsilon}\right)}_{I(\epsilon)}\\
&\leq \epsilon\log_b(M-1)+w\log_b(2/w)\leq \epsilon\log_b(M-1)+2e^2\log_b(e^2).\end{split}\end{equation} To verify the first inequality above, simply check that the function $I(\epsilon)$ has a unique critical point at $\epsilon=w/2(M-1)$. To verify the second, check that $w\log_b(2/w)$ has a unique critical point at $w=2e^{-2}$. Since $I(\epsilon)\geq 0$, the first line in \eqref{entropy-expansion} and \eqref{alpha-bounds} yield: $$\gamma=1-L\alpha-\epsilon \geq 1-L\alpha-\frac{s-L\alpha\log_b(1/\alpha)}{\log_b(M-1)}\geq 1-\frac{\min(1,s^2)}{e^2}-\left(1-\frac{1}{e}\right)\frac{s}{\log_b(M-1)}.$$ This confirms \eqref{gamma-bound}. To check \eqref{gamma-bound-2}, note that $\gamma=w-(M-1)\beta\geq w/M$ (since $\beta\leq w/M$) and recall that $w=1-L\alpha\geq 1-t^2/e^2$. Finally, if $2e^2\log_b(e^2)\leq(\frac{1}{2}-\frac{1}{e})s$, then \eqref{entropy-expansion} and \eqref{alpha-bounds} yield \eqref{beta-bound}: \begin{equation*}\beta=\frac{\epsilon}{M-1}
\geq \frac{s-L\alpha\log_b(1/\alpha)-2e^2\log_b(e^2)}{(M-1)\log_b(M-1)}\geq \frac{s}{2(M-1)\log_b(M-1)}.\qedhere\end{equation*}
\end{proof}

\begin{figure}[t]\includegraphics[width=.44\textwidth]{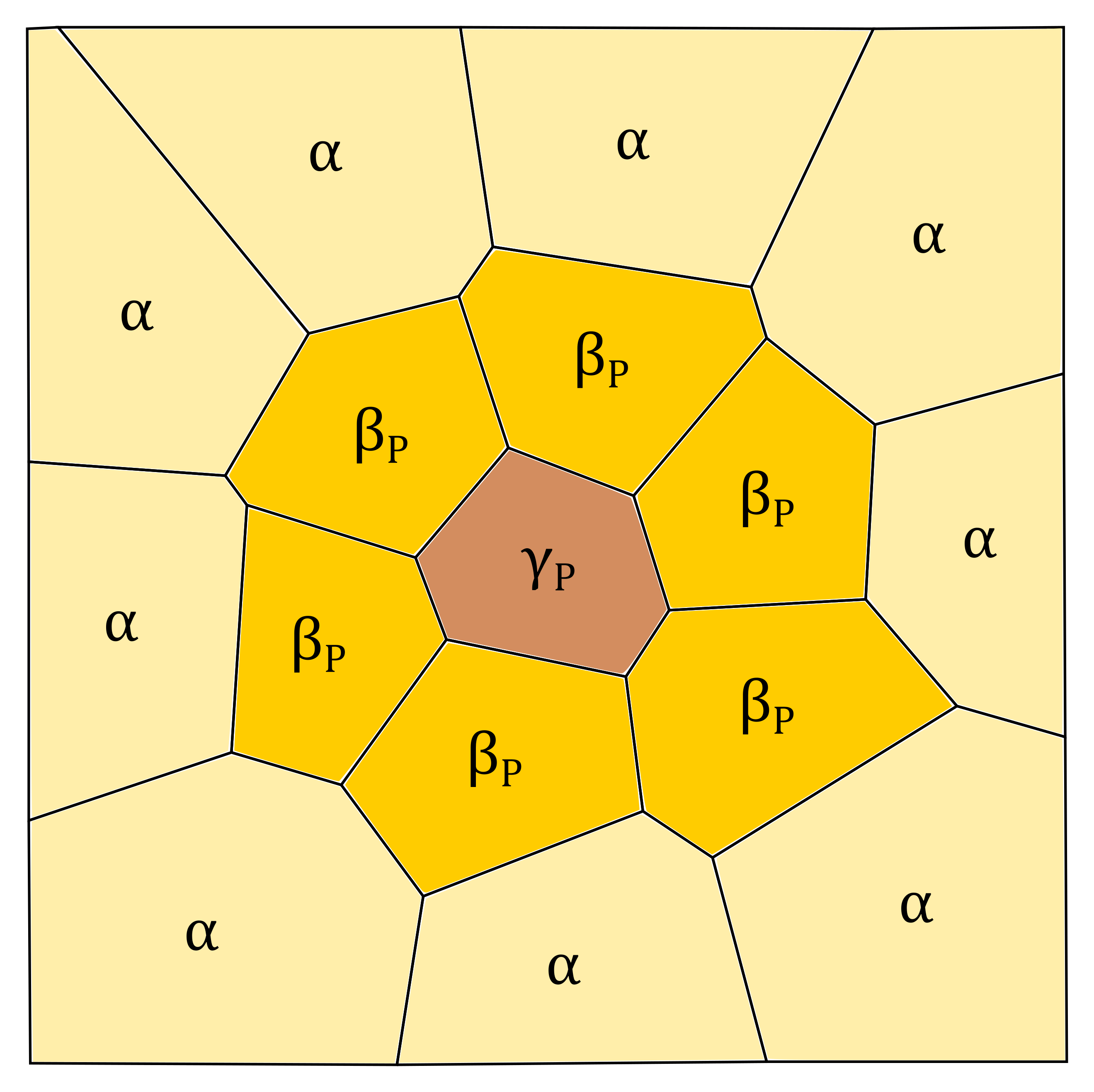}\hspace{.08\textwidth}\includegraphics[width=.44\textwidth]{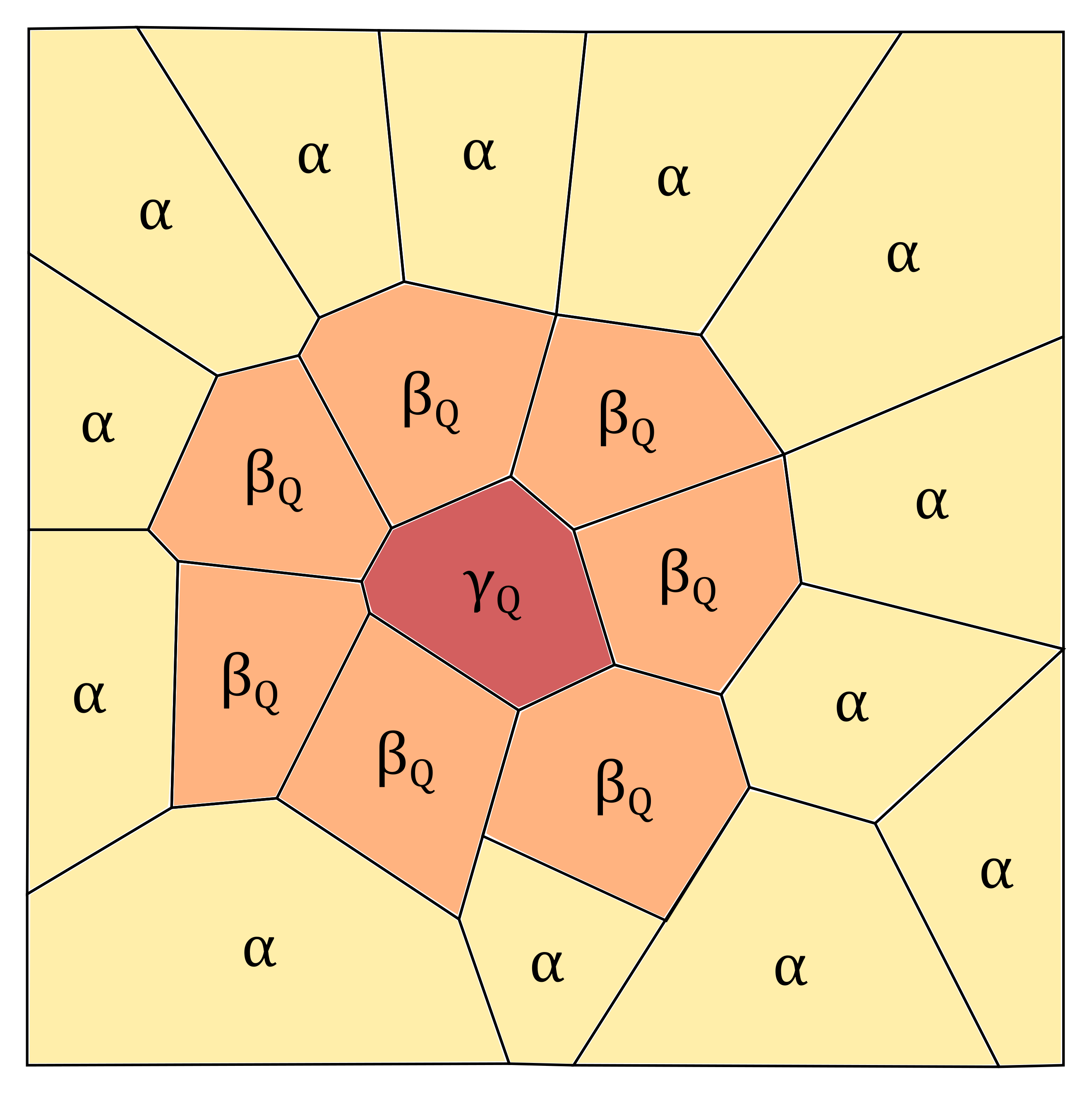}
\caption{On each metric cube $Q\in\Delta_k$ with $k\geq 0$, we try to force the quasi-Bernoulli measure $\muse$ to look locally $s_{k+1}$-dimensional by distributing the mass of the cube to its children so that the central child receives the most mass of any child. The number of children of a cube can fluctuate throughout the space. To get a doubling measure, we choose the weight $\alpha$ of children in $\Outer(Q)$ independently of $Q$. The weight $\beta_Q$ of children in $\Inner(Q)\setminus\{Q^\downarrow\}$ and weight $\gamma_Q$ of the central child $Q^\downarrow$ depend on $\alpha$, $q$, $s_{k+1}$, $L_Q=\#\Outer(Q)$, and $M_Q=\#\Inner(Q)$. In the graphic, we illustrate mass distributions for two cubes $P,Q\in\Delta_k$ with $\muse(P)=\muse(Q)=1$. On the left $L_P=9$ and $M_P=7$, while on the right $L_Q=13$ and $M_Q=8$.}\label{fig:qb} \end{figure}

\begin{definition}[quasi-Bernoulli measures] \label{regime-measures} Let $\XX$ be a complete Ahlfors $q$-regular metric space with $\diam\XX\geq 2.1$, let $\nu$ be a doubling measure on $\XX$, and let $\mathbf{s}=(s_k)_{k=1}^\infty$ be a sequence of positive numbers (``target dimensions'') such that \begin{equation}s_*:=\inf_{k\geq 1}s_k>0\quad\text{and}\quad s^*:=\sup_{k\geq 1} s_k<q.\end{equation} Let $(\Delta_k)_{k\in\ZZ}$ be a system of $b$-adic cubes for $\XX$ for some large $b\geq 47$. For all  $Q\in\Delta$, assign $L_Q:=\#\Outer(Q)$, $M_Q:=\#\Inner(Q)$, and $N_Q:=\#\Child(Q)$. We require that $b$ be large enough depending on at most $\XX$ and $s^*$ so that \begin{equation}\label{LM-bounds}M_Q\geq b^{s^*}\quad\text{and}\quad L_Q\leq N_Q\leq b^{q+1}\quad\text{for all $Q\in\Delta_+=\bigcup_{k=0}^\infty \Delta_k$.}\end{equation} (See Corollary \ref{c:good-count}.) Let $0<\alpha\leq \left(\frac12\min\{s_*,1\}\ln(b)b^{-(q+1)}\right)^2$ be a given weight; cf.~\eqref{alpha-def}. For all $k\geq 0$ and $Q\in\Delta_k$, we may use Lemma \ref{entropy-2} to define unique weights $$\beta_Q=\beta(\alpha,b,q+1,s_{k+1},L_Q,M_Q)\quad\text{and}\quad \gamma_Q=\gamma(\alpha,b,q+1,s_{k+1},L_Q,M_Q)$$ satisfying \begin{equation}\label{qb-weights} 1=L_Q\alpha+(M_Q-1)\beta_Q+\gamma_Q\quad\text{and}\quad h_{b,L_Q,M_Q}(\alpha,\beta_Q)=s_{k+1}.\end{equation} We specify a Radon measure $\muse$ on $\XX$ by specifying its values on cubes as follows: \begin{enumerate}
\item Declare $\muse(Q):=\nu(Q)$ for all $Q\in\Delta_0$.
\item For all $k\geq 0$ and $Q\in\Delta_k$, declare $\muse(R):=\alpha\muse(Q)$ for all $R\in\Outer(Q)$, declare $\muse(R):=\beta_Q\muse(Q)$ for all $R\in\Inner(Q)\setminus\{Q^\downarrow\}$, and declare $\muse(Q^{\downarrow}):=\gamma_Q\muse(Q)$.
\end{enumerate} We call $\muse$ a \emph{quasi-Bernoulli measure on} $\XX$ \emph{with target dimensions} $\mathbf{s}$, \emph{background measure} $\nu$, and \emph{outer weight} $\alpha$. (Of course, $\muse$ also depends on the choice of $b$ and $(\Delta_k)_{k\in\ZZ}$.)
\end{definition}

\begin{lemma}[existence and doubling] \label{qb-doubling} For any sequence $\mathbf{s}=(s_k)_{k\geq 1}$ of target dimensions, background measure $\nu$, and outer weight $\alpha$, the quasi-Bernoulli measure $\muse$ exists and \begin{equation}\label{support} 0<\muse(B(x,r))<\infty\quad\text{for all $x\in\XX$ and $r>0$}.\end{equation} If $b$ is large enough so that $2e^2\log_b(e^2)\leq s_*(\tfrac12-\tfrac{1}{e})$,  then $\muse$ is doubling and \begin{equation} \label{b-doubling} \muse(B(x,2r))\lesssim_{\nu,\alpha,b,s_*,q}\muse(B(x,r))\quad\text{for all $x\in\XX$ and $r>0$,}\end{equation} where the dependence on $\nu$ is on the doubling constant of $\nu$.\end{lemma}

\begin{proof} \emph{Existence and uniqueness.} On the one hand, if $Q\in\Delta_+$ and $\log_b(M_Q-1)\geq 2s^*$, then $\muse(Q^\downarrow)=\gamma_Q\muse(Q)\geq (\frac{1}{2}-\frac{1}{e^2}+\frac{1}{2e})\muse(Q)$ by \eqref{gamma-bound}. On the other hand, if $\log_b(M_Q-1)\leq 2s^*$, then $M_Q\leq 1+b^{2s^*}$ and $\muse(Q^\downarrow)=\gamma_Q\muse(Q)\gtrsim_{b,s^*}\muse(Q)$ by \eqref{gamma-bound-2}. Therefore, by \eqref{qb-weights} and Remark \ref{practical}, the measure $\muse$ exists and is the unique Radon measure taking the indicated values on $\Delta_+$. Because $0<\muse(Q)<\infty$ for all $Q\in\Delta$, every ball in $\XX$ contains some cube in $\Delta$, and every ball is contained in some cube in $\Delta$ (see Remark \ref{r:exhaustive}), the measure $\muse$ has full support \eqref{support}.

\emph{Doubling on large radii.} On balls with large radii, $\muse$ inherits the doubling property from $\nu$. Let $x\in\XX$ and suppose that $2C_b b^{-k}\leq r<2C_b b^{-(k-1)}$ for some integer $k\leq 0$. On the one hand, if we let $Q\in \Delta_k$ be the unique cube of side length $b^{-k}$ containing $x$, then $$\muse(B(x,r))\geq \muse(Q)=\nu(Q)\geq \nu(U_Q)\geq \nu(B(x_Q,\tfrac{1}{2}c_b b^{-k})).$$ On the other hand, let $\mathcal{R}=\{R\in\Delta_0:R\cap B(x,2r)\neq\emptyset\}$. For any $R\in\mathcal{R}$ and $y\in R$, $$|y-x_Q|\leq |y-x|+|x-x_Q|\leq (\diam R+2r)+r\leq 7C_b b^{-(k-1)}=(14bC_b/c_b)\tfrac12c_b b^{-k}.$$ Hence $$\muse(B(x,2r))\leq \sum_{R\in\mathcal{R}} \muse(R) = \sum_{R\in\mathcal{R}} \nu(R) \leq \nu(B(x_Q,(14b C_b/c_b) \tfrac{1}{2}c_bb^{-k})).$$ It follows that $$\muse(B(x,2r)) \leq \nu(B(x_Q,(14b C_b/c_b) \tfrac{1}{2}c_bb^{-k})) \lesssim_{\nu,b} \nu(B(x_Q,\tfrac{1}{2}c_b b^{-k}))\leq \muse(B(x,r)),$$ where the implicit constant depends only on the doubling constant for $\nu$ and $b$.

\emph{$b$-adic doubling.} Assume that $2e^2\log_b(e^2)\leq s_*(\tfrac12-\tfrac{1}{e})$. Then $\muse$ is $b$-adic doubling in the sense that \begin{equation}\label{b-adic-doubling}\muse(R)\gtrsim_{\nu,\alpha,b,s_*,q}\muse(Q)\quad\text{for all $Q\in\Delta$ and $R\in\Child(Q)$}.\end{equation} There are two regimes. First, if $\side Q=b^{-k}$ for some $k\leq -1$, then $$\muse(R)=\nu(R)\gtrsim_{\nu,b}\nu(Q)=\muse(Q)$$ by an argument similar to the one in the previous paragraph, since $\nu$ is doubling. Second, if $\side Q=b^{-k}$ for some $k\geq 0$, then either $\muse(R)=\alpha\muse(Q)$ or $\muse(R)=\beta_Q\muse(Q)$ or $\muse(R)=\gamma_Q\muse(Q)$, so it suffices to find a common lower bound for all of the weights. In the existence paragraph, we already showed that $\gamma_Q\gtrsim_{b,s^*}1$. Since $2e^2\log_b(e^2)\leq s_*(\tfrac12-\tfrac{1}{e})\leq s_{k+1}(\tfrac12-\tfrac{1}{e})$, we also have $\beta_Q\gtrsim_{b,s_*,q}1$ by \eqref{beta-bound} and \eqref{LM-bounds}. Finally, the weight $\alpha>0$ by definition. Therefore, $\muse$ is $b$-adic doubling.

\emph{Doubling on small radii.} The key point is that the weights $\alpha$ attached to children $R$ of $Q\in\Delta_k$ ($k\geq 0$) that lie near $\partial Q$ do not depend on $Q$. This idea is already present in \cite[\S3]{KRS-cubes} and \cite[\S8]{BLZ}. We continue to assume that $2e^2\log_b(e^2)\leq s_*(\tfrac12-\tfrac{1}{e})$.

Let $x\in\XX$ and suppose that $2C_b b^{-k}\leq r< 2C_b b^{-(k-1)}$ for some integer $k\geq 1$. Let $Q$ be the unique cube of side length $b^{-k}$ containing $x$ and let $\mathcal{S}=\{R\in\Delta_{k-1}:R\cap B(x,2r)\neq\emptyset\}$. It is easy to see using the doubling property of $\nu$ that $\#\mathcal{S}\lesssim_{\nu,b} 1$ (cf.~proof of Lemma \ref{l:AR-count}). Thus, if we can show that $\muse(R)\lesssim_{\nu,\alpha,b,s_*,q} \muse(S)$ for all $R,S\in\mathcal{S}$, then $$\muse(B(x,2r))\leq \sum_{R\in\mathcal{S}} \muse(R) \lesssim_{\nu,b,s_*,q}\muse(Q^\uparrow) \lesssim_{\nu,\alpha,b,s_*,q}\muse(Q)\leq  \muse(B(x,r)).$$ Fix $R,S\in\mathcal{S}$ with $R\neq S$. Since $R$ and $S$ both intersect $B(x,2r)$, we have $\gap(R,S)\leq \diam B(x,2r)\leq 8C_b\, b^{-(k-1)}< 9\side R$, since $b\geq 47$ implies $C_b\leq 47/46$. To proceed, let us assign labels to the ancestors of $R$ and $S$: for all $0\leq j\leq k-1$, let $Q^R_j$ and $Q^S_j$ be unique cubes in $\Delta_j$ containing $R$ and $S$, respectively. Note that $Q^R_{k-1}=R$ and $Q^S_{k-1}=S$. Let $i\geq 0$ be the least index such that $Q^R_i\neq Q^S_i$. If $i=0$, then $\muse(Q^R_0)=\nu(Q^R_0)$, $\muse(Q^S_0)=\nu(Q^S_0)$, and $\gap(Q^R_0,Q^S_0)\leq \gap(R,S)<9$, so \begin{equation}\label{0-doubling}\muse(Q^R_0)\lesssim_\nu \muse(Q^S_0)\lesssim_{\nu}\muse(Q^R_0)\end{equation} by the doubling of $\nu$. If $i\geq 1$, then $(Q^R_i)^\uparrow = Q^R_{i-1}=Q^S_{i-1}=(Q^S_i)^\uparrow$. Hence \begin{equation}\label{1-doubling}
\muse(Q^R_i)\leq \muse(Q^R_{i-1})=\muse(Q^S_{i-1})\lesssim_{\nu,\alpha,b,s_*,q}\muse(Q^S_{i})\end{equation} by \eqref{b-adic-doubling}. The same argument shows $\muse(Q^S_i)\lesssim_{\nu,\alpha,b,s_*,q}\muse(Q^R_i)$. For all $i+1\leq j\leq k-1$, we have $\gap(Q^R_j,\XX \setminus Q^R_{j-1})\leq \gap(R,S)<9 \side R  \leq  9\side Q^R_j$, whence $Q^R_j\in\Outer(Q^R_{j-1})$ by Lemma \ref{l:2AR-count}. The same is true if we swap the role of $R$ and $S$. Thus, $\muse(Q^R_j)=\alpha\muse(Q^R_{j-1})$ and $\muse(Q^S_j)=\alpha\muse(Q^S_{j-1})$ for all $i+1\leq j\leq k-1$. It follows that \begin{equation} \muse(R)=\alpha^{k-1-i} \muse(Q^R_i) \lesssim_{\nu,\alpha,b,s_*,q} \alpha^{k-1-i} \muse(Q^S_i)=\muse(S).\end{equation} Similarly, $\muse(S)\lesssim_{\nu,b,s_*,q}\muse(R)$. Therefore, $\muse$ is a doubling measure.
\end{proof}

\begin{lemma}[dimension]\label{qb-dimension} If $\muse$ is $b$-adic doubling (e.g.~if $2e^2\log_b(e^2)\leq s_*(\tfrac12-\tfrac{1}{e})$), then \begin{equation}\label{qb-dims} \underline{\dim}_H\,\muse=\overline{\dim}_H\,\muse=\liminf_{k\rightarrow\infty} \frac{1}{k}\sum_{j=1}^k s_j,\quad\underline{\dim}_P\,\muse=\overline{\dim}_P\,\muse=\limsup_{k\rightarrow\infty}\frac{1}{k}\sum_{j=1}^k s_j.\end{equation}\end{lemma}

\begin{proof} To prove \eqref{qb-dims}, we may fix $Q_0\in\Delta_0$ and show that (see \S\ref{ss:dimension}) $$\liminf_{k\rightarrow\infty} \frac{\log\muse(Q_k(x))}{\log(b^{-k})}=\liminf_{k\rightarrow\infty} \frac{1}{k}\sum_{j=1}^k s_j,\quad \limsup_{k\rightarrow\infty} \frac{\log\muse(Q_k(x))}{\log(b^{-k})}=\limsup_{k\rightarrow\infty} \frac{1}{k}\sum_{j=1}^k s_j$$ at $\muse$-a.e.~$x\in Q_0$, where $Q_k(x)\in\Delta_k$ is the cube containing $x$. For convenience, so that we may adopt the probabilistic viewpoint, we temporarily rescale $\muse$ so that $\muse(Q_0)=1$. Let $\EE$ denote the expectation with respect to the probability measure $\PP:=\muse\res Q_0$. Following the same plan as in \cite[\S3]{dimension-of-measures}, for each $k\geq 1$ we define a random variable $X_k$ on $Q_0$ by \begin{equation*}X_k(x)=-\log_b \muse(Q_k(x)) + \log_b \muse(Q_{k-1}(x)).\end{equation*} Note that $|X_k(x)|\lesssim_{\muse} 1$ for all $x\in Q_0$, since $\muse$ is $b$-adic doubling. The relevance of these random variables for the problem at hand is that \begin{equation}\label{S-average}\frac{S_k(x)}{k}:=\frac{X_1(x)+\cdots+X_k(x)}{k}=\frac{\log \muse(Q_k(x))}{\log (b^{-k})}\quad\text{for all }x\in Q_0,\end{equation} since $\log_b(\muse (Q_0(x))=\log_b(\muse(Q_0))=0$ by our assumption that $\muse(Q_0)=1$.

We claim that the random variables $X_k$ are uncorrelated, i.e.~$\EE(X_iX_j)=\EE(X_i)\EE(X_j)$ for all $i\neq j$. For each $j\geq 1$, $X_j$ is constant on each cube $R\in\Child^j(Q_0)$ and the value $X_j(R)$ that it takes can be computed knowing only type of child that $R$ is: \begin{equation}X_j(R)=\left\{\begin{array}{ll}\log_b(1/\alpha) &\text{when } R\in\Outer(R^\uparrow),\\ \log_b(1/\beta_Q) &\text{when } R\in\Inner(R^\uparrow)\setminus\{R^{\uparrow\downarrow}\},\\ \log_b(1/\gamma_Q) &\text{when } R=R^{\uparrow\downarrow}.\end{array}\right.\end{equation} Thus, by \eqref{abc-2} and \eqref{qb-weights}, we obtain \begin{equation}\label{X-expectation} \EE(X_j)= \!\!\!\sum_{Q\in\Child^{j-1}(Q_0)} \sum_{R\in\Child(Q)} \!\!\! X_j(R)\mu(R) =\!\!\!\sum_{Q\in\Child^{j-1}(Q_0)} \!\!\! h_{b,L_Q,M_Q}(\alpha,\beta_Q)\,\muse(Q)=s_j.\end{equation}  A similar computation shows that $\EE(X_iX_j)=s_is_j$ when $1\leq i<j$. Indeed, since $X_i$ takes constant values on each cube in $\Delta_i$ and $X_j$ takes constant values on each cube in $\Delta_j$, and $j>i$, we witness that \begin{equation}\begin{split}
\EE(X_iX_j)&=\sum_{P\in\Child^{i-1}(Q_0)}\sum_{Q\in\Child(P)}\sum_{R\in\Child^{j-1-i}(Q)}\sum_{S\in\Child(R)} X_i(Q)X_j(S)\muse(S)\\
&=\sum_{P\in\Child^{i-1}(Q_0)}\sum_{Q\in\Child(P)}X_i(Q)\sum_{R\in\Child^{j-1-i}(Q)} h_{b,L_R,M_R}(\alpha,\beta_R) \muse(R)\\
&=\sum_{P\in\Child^{i-1}(Q_0)}\sum_{Q\in\Child(P)}X_i(Q)\muse(Q)s_j\\
&=s_is_j.
\end{split}\end{equation}

Because the random variables $X_k$ are uncorrelated and are uniformly bounded in $L^2$ (as they are uniformly bounded in $L^\infty$), the strong law of large numbers (see e.g.~\cite[Theorem 5.1.2]{Chung-probability}) ensures that \begin{equation}\label{LLN} \lim_{k\rightarrow\infty} \frac{S_k-\EE(S_k)}{k}=0\quad\text{at $\muse$-a.e.~$x\in Q_0$}.\end{equation} Combining \eqref{S-average}, \eqref{X-expectation}, and \eqref{LLN}, we conclude that at $\muse$-a.e.~$x\in Q_0$, \begin{align*}
\liminf_{k\rightarrow\infty} \frac{\log \muse(Q_k(x))}{\log (b^{-k})}
   &=\liminf_{k\rightarrow\infty} \frac{1}{k}\sum_{j=1}^k \EE(X_j)
    = \liminf_{k\rightarrow\infty} \frac{1}{k}\sum_{j=1}^\infty s_j\quad\text{and}\\
\limsup_{k\rightarrow\infty} \frac{\log \muse(Q_k(x))}{\log (b^{-k})}
   &=\limsup_{k\rightarrow\infty} \frac{1}{k}\sum_{j=1}^k \EE(X_j)
    = \limsup_{k\rightarrow\infty} \frac{1}{k}\sum_{j=1}^\infty s_j.\qedhere\end{align*}
\end{proof}

It is now an easy matter to note the existence of doubling measures with prescribed Hausdorff and packing dimensions; see \cite{Wu-dimension}, \cite{Bylund-Gudayol}, \cite{KRS-cubes}, \cite{no-AR-subsets} for related results on complete doubling metric spaces, which always support a doubling measure \cite{doubling-exists}, \cite{VK-doubling}. While the authors would not be surprised to find that Theorem \ref{t:prescribe} is already known, they could not find a reference for it in the literature.

\begin{theorem}\label{t:prescribe} Let $\XX$ be a complete Ahlfors $q$-regular metric space. For any four numbers $d_*$, $d^*$, $D_*$, $D^*$ satisfying $0<d_*\leq d^*\leq q$, $0<D_*\leq D^*\leq q$, $d_*\leq D_*$, and $d^*\leq D^*$, and $d_*=q\Leftrightarrow D_*=q$ and $d^*=q\Leftrightarrow D^*=q$, there is a doubling measure $\mu$ on $\XX$ such that \begin{equation}\label{haus-pack-dims}\underline{\dim}_H\,\mu=d_*,\quad \overline{\dim}_H\,\mu=d^*,\quad \underline{\dim}_P\,\mu=D_*,\quad \overline{\dim}_{P}\,\mu=D^*.\end{equation}\end{theorem}

\begin{proof} If $\diam \XX=0$, the conclusion is trivial. Thus, after scaling the metric if necessary, we may suppose that $\diam\XX\geq 2.1$. Because the sum of two doubling measures is again a doubling measure, it suffices to prove that for any $0<d\leq D\leq q$ with $d=q\Leftrightarrow D=q$, there exists a doubling measure $\mu$ on $\XX$ such that $\underline{\dim}_H\,\mu=\overline{\dim}_H\,\mu=d$ and $\underline{\dim}_P\,\mu=\overline{\dim}_P\,\mu=D$. If $d=D=q$, then we may simply take $\mu$ to be any Ahlfors $q$-regular measure on $\XX$. Otherwise, $0<d\leq D<q$. Choose any sequence $\mathbf{s}=(s_k)_{k=1}^\infty$ taking values in $\{d,D\}$ such that $\liminf_{k\rightarrow\infty} k^{-1}(s_1+\cdots+s_k)=d$ and $\limsup_{k\rightarrow\infty} k^{-1}(s_1+\cdots+s_k)=D$. Let $\mu=\muse$ be a quasi-Bernoulli measure with target dimensions $\mathbf{s}$, built using a system of $b$-adic cubes with scaling factor $b$ satisfying $2e^2\log_b(e^2)\leq s_*(\tfrac12-\tfrac{1}{e})=d(\tfrac{1}{2}-\tfrac{1}{e})$. Then $\mu$ exists and is doubling by Lemma \ref{qb-doubling} and has Hausdorff dimension $d$ and packing dimension $D$ by Lemma \ref{qb-dimension}.\end{proof}

Because the construction of the quasi-Bernoulli measure $\muse$ demands that we take $b\uparrow\infty$ as $s^*\uparrow q$ (see Corollary \ref{c:good-count}), we do not currently know how to build a doubling measure $\mu$ on a complete Ahlfors $q$-regular space such that $\underline{\dim}_H\,\mu=\overline{\dim}_H\,\mu<q=\underline{\dim}_P\,\mu=\overline{\dim}_P\,\mu$. Despite this technical obstruction, we cannot think of a compelling reason why such a measure should not exist.

\begin{conjecture} \label{all-d} Theorem \ref{t:prescribe} also holds without the requirement that $d_*=q\Leftrightarrow D_*=q$ and $d^*=q\Leftrightarrow D^*=q$.\end{conjecture}

To establish the rectifiability of quasi-Bernoulli measures (see \S\ref{s:rect}), we need two more estimates.
For all $Q\in\Delta$ and for all integers $n\geq 1$ and $0\leq k\leq n$, define \begin{equation} \label{KQ}
K_Q(n,k):=\{R\in\Child^n(Q):\mathsf{cc}(Q,R)\geq n-k\},\end{equation} where $\mathsf{cc}(Q,R)$ is given by \eqref{cc-def}. That is, the collection $K_Q(n,k)$ consists of all $n$-th generation descendants of $Q$ such that at least $n-k$ members in the chain of descendants between $Q$ and $R$ were the central child of their parent. Note that $\#K_Q(n,0)=1$, $K_Q(n,n)=\Child^n(Q)$, and $K_Q(n,k-1)\subset K_Q(n,k)$ for all $1\leq k\leq n$.

\begin{lemma}[measure concentration] \label{l:measure-t} Let $Q\in\Delta$, let $0<t<1$, and suppose that $\muse(R^\downarrow)\geq (1-t)\muse(R)$ for all $R\in\bigcup_{j=0}^{n-1}\Child^j(Q)$, where $\Child^0(Q)=\{Q\}$. If $n$ and $k$ are positive integers and $tn\leq k\leq n$, then \begin{equation}\label{measure-bound-t} \muse\left(\textstyle\bigcup K_Q(n,k)\right)\geq\muse\left(\textstyle\bigcup K_Q(n,k-1)\right)\geq \left(1-\exp\left(-\frac{n}{2}\left[\frac{k}{n}-t\right]^2\right)\right)\muse(Q).\end{equation}\end{lemma}

\begin{proof} The first inequality in \eqref{measure-bound-t} is trivial and is recorded for the convenience of applying the lemma in \S\ref{s:rect}. As to the main matter, we will derive the second inequality in \eqref{measure-bound-t} from Azuma's inequality, a standard measure concentration estimate for martingales with bounded differences (see e.g.~\cite[top of p.~96]{concentration-survey}).
Let $\mathbb{P}$ denote the Borel probability measure $(\muse(Q))^{-1}\muse\res Q$ and let $\mathbb{E}$ denote its expectation. For each integer $1\leq j\leq n$, define the random variable \begin{equation*}Y_j:=\sum_{R\in\Child^{j}(Q)} \mathsf{ci}(R)\chi_R=\sum_{S\in\Child^{j-1}(Q)}\chi_{S^\downarrow}.\end{equation*} That is, $Y_j$ is the sum of indicator functions for $j$-th generation descendants of $Q$ that are the central child of their parents. By the hypothesis on the measure of central children, we have $1-t\leq \mathbb{E}(Y_j)\leq 1$ for all $j\geq 1$ and $$Q\setminus \textstyle\bigcup K_Q(n,k-1)=\{x\in Q:\sum_{j=1}^{n} Y_j(x)\leq n-k\}.$$ Next, define random variables $X_0\equiv 0$ and $X_j:=\sum_{i=1}^j \left(Y_i-\mathbb{E}(Y_i)\right)$ for all $j\geq 1$. Then the sequence $X_0,X_1,\cdots$ is a martingale with respect to the filtration $(\mathcal{F}_j)_{j=0}^\infty$, where $\mathcal{F}_j$ is the $\sigma$-algebra generated $\Child^j(Q)$. Moreover, for all $j\geq 1$, $c_j:=\|X_{j}-X_{j-1}\|_\infty = \|Y_{j}-\mathbb{E}(Y_{j})\|_\infty\leq 1$. Suppose that $k-nt\geq 0$. Then, by Azuma's inequality, \begin{equation*}\begin{split}
\mathbb{P}\left(\textstyle\sum_{j=1}^{n} Y_j\leq n-k\right)&=\mathbb{P}\left(X_{n}\leq n-k-\textstyle\sum_{j=1}^{n}\mathbb{E}(Y_j)\right)\\
&\leq \mathbb{P}\left(X_{n}-X_0\leq -\left[k-nt\right]\right)\\
&\stackrel{\textrm{Azuma}}{\leq} \exp\left( \frac{-\left[k-nt\right]^2}{2\sum_{j=1}^{n}c_j^2}\right)\leq\exp\left(-\frac{n}{2}\left[\frac{k}{n}-t\right]^2 \right).\end{split}\end{equation*} This yields \eqref{measure-bound-t}.
\end{proof}

\begin{lemma}[cardinality estimate] \label{l:KQ-bound} Let $Q\in\Delta$ and let $n\geq 1$. Suppose that $\hat N$ is a positive number such that \begin{equation}\label{hatN-requirement}\max\left\{N_R:0\leq i\leq n-1,\;R\in\Child^{i}(Q)\right\}\leq \hat N,\quad\text{where }N_R=\#\Child(R).\end{equation}
If $n$ and $k$ are positive integers with $k\leq n\hat N/(\hat N + 1)$, then \begin{equation}\label{KQ-bound}\#K_Q(n,k) \leq \hat N^k\frac{n^n}{k^k(n-k)^{n-k}}\leq b^{n\left((k/n)\log_b(\hat N)+\log_b(2)\right)}\quad\text{for all $b>1$}.\end{equation} \end{lemma}

\begin{proof} Let $1\leq k\leq n$. Since $K_Q(n,k)=\bigcup_{i=n-k}^{n} \{R\in\Child^n(Q):\mathsf{cc}(Q,R)=i\}$, for any real-valued $x\geq 1$, we may bound $$\#K_Q(n,k) \leq \sum_{i=n-k}^n\binom{n}{i} \hat N^{n-i}=\sum_{j=0}^k \binom{n}{j}\hat N^j\leq \sum_{j=0}^n \binom{n}{j} \hat N^j x^{k-j}=x^{k-n}(x+\hat N)^n=:f(x).$$ The function $f:(0,\infty)\rightarrow\RR$ has a unique critical point \begin{equation*}x_0=\frac{\hat N(n-k)}{k}\end{equation*} and $x_0\geq 1$ precisely when $k \leq n\hat N/(\hat N+1)$. Since $\lim_{x\rightarrow 0+} f(x)=\lim_{x\rightarrow\infty} f(x)=\infty$, the first upper bound on $\#K_Q(n,k)$ in \eqref{KQ-bound} is achieved by choosing $x=x_0$. The second bound in \eqref{KQ-bound} follows by expanding $k=(k/n)n$, writing $x=b^{\log_b(x)}$, and applying the entropy bound $\epsilon\log_b(1/\epsilon)+(1-\epsilon)\log_b(1/(1-\epsilon))\leq \log_b(2)$ for all $0<\epsilon<1$.
\end{proof}

\begin{remark}Lemma \ref{l:measure-t} requires a lower bound $k\gtrsim_t n$ whereas Lemma \ref{l:KQ-bound} requires an upper bound $k\lesssim_{\hat N} n$.\end{remark}

\section{Proof of Theorem \ref{t:main}}\label{s:rect}

Suppose that $\XX$ is a complete Ahlfors $q$-regular metric space, $m\geq 1$ is an integer with $q>m-1$, and $0<s_H\leq s_P<q$ are real numbers with $m-1<s_P< m$. Without loss of generality, we may assume that $\diam\XX\geq 2.1$. Let $\nu$ be any doubling measure on $\XX$. Let $b\geq 47$ be large and fixed, ultimately depending on at most $q$, $s_H$, $s_P$, and $\XX$. Let $(\Delta_k)_{k\in\ZZ}$ be any system of $b$-adic cubes for $\XX$ (see \S\ref{ss:cubes}). By Lemma \ref{l:AR-count} and Lemma \ref{l:2AR-count}, there exists $C_\XX>1$ such that for all $Q\in\Delta_+=\bigcup_{k=0}^\infty\Delta_k$ and $j\geq 1$, \begin{equation}\label{card-bound} C_\XX^{-1} b^{jq}\leq \#\Child^j(Q)\leq C_\XX b^{jq}\quad\text{and}\quad \#\Inner(Q)\geq C_\XX^{-1} b^q.\end{equation} Let $\mathbf{s}=(s_k)_{k=1}^\infty$ be a sequence of numbers taking values in $[s_H,s_P]$ such that \begin{align}
s_*&=\inf_{k\geq 1}s_k=s_H=\liminf_{k\rightarrow\infty} \frac{1}{k}\sum_{j=1}^k s_j, \label{rect-liminf}\\
s^*&=\sup_{k\geq 1}s_k=s_P=\limsup_{k\rightarrow\infty} \frac{1}{k}\sum_{j=1}^k s_j. \label{rect-limsup}\end{align} Let $\alpha>0$ be small and fixed, ultimately depending on at most $q$, $s_H$, $s_P$, $b$, and $\XX$. Finally, let $\muse$ be the quasi-Bernoulli measure with target dimensions $\mathbf{s}$, background measure $\nu$, and outer weight $\alpha$ as in Definition \ref{regime-measures}. To prove the theorem, we show that \eqref{doubling}--\eqref{bi-lip-vanish} hold with $\mu=\muse$. A lot of the work has been done already. For large enough $b$, the measure $\muse$ is doubling in the sense of \eqref{doubling} by Lemma \ref{qb-doubling} and $\muse$ satisfies \eqref{haus-dim-s} and \eqref{pack-dim-s} by Theorem \ref{metric-local-dim}, Lemma \ref{qb-dimension}, \eqref{rect-liminf}, and \eqref{rect-limsup}. As we already noted in the introduction, $\muse$ is purely $(m-1)$-unrectifiable in the sense of \eqref{unrectifiable}, because $\underline{\dim}_P\,\muse=s_P>m-1$ and Lipschitz images of subsets of $\RR^{m-1}$ have packing dimension at most $m-1$. Finally, we will prove \eqref{m-rectifiable} in \S\ref{ss:simple} and prove \eqref{bi-lip-vanish} in \S\ref{ss:porous}.

\subsection{Rectifiability}\label{ss:simple}
We would like to prove that $\muse$ is  $m$-rectifiable in the sense of \eqref{m-rectifiable}. Since $\Delta_0$ is countable and covers $\XX$, it suffices to fix $Q_0\in\Delta_0$ and prove that $\muse\res Q_0$ is $m$-rectifiable. Further, it is enough to fix $0<\theta<1$ (close to 1) and find a tree of sets $\mathcal{T}=\bigcup_{l=0}^\infty \mathcal{T}_l$ such that $\muse(\leaves(\mathcal{T}))>\theta\muse(Q_0)$ and $\leaves(\mathcal{T})$ lies in the image of some Lipschitz map $f:E\subset[0,1]^m\rightarrow\XX$.

\emph{Definition of $\mathcal{T}$.} For the remainder of \S\ref{ss:simple}, we fix any number $\tau=\tau(q,m,s_P)>0$ satisfying  \begin{equation}\label{tau61} (1+3\tau)\frac{s_P}{m}<1\quad\text{and}\quad a:=(1+2\tau)\frac{s_P}{q}\in\mathbb{Q}\cap(0,1).\end{equation} Write $a=u/v$ in reduced form, i.e.~$u,v\in\ZZ$, $v\geq 1$, and $\mathrm{gcd}(u,v)=1$. Let $n_0\geq v$ be a large integer such that $n_0\equiv 0\pmod{v}$. For all $i\geq 1$, we put $n_i:=in_0$ and $k_i:=an_i$. Since $v$ divides $n_i$, we get that $k_i$ is also an integer. We build $\mathcal{T}$ using induction. Initialize $\mathcal{T}_0=\{Q_0\}$. Then, assuming that the level $\mathcal{T}_j$ of $\mathcal{T}$ has been defined for some $j\geq 0$, define the next level $\mathcal{T}_{j+1}$ by specifying that \begin{equation}\label{next-level} \Child_\mathcal{T}(Q)\cap \mathcal{T}_{j+1} = K_Q(n_{j+1},k_{j+1})\quad\text{for all }Q\in\mathcal{T}_j,\end{equation} with $K_{Q}(n,k)$ as in \eqref{KQ}. Note that \begin{equation}\label{art-side}\side Q = b^{-(n_1+\cdots+n_j)}=b^{-\frac{1}{2}j(j+1)n_0}\quad\text{for any $j\geq 0$ and $Q\in\mathcal{T}_j$}.\end{equation}

\emph{We show that $\leaves(\mathcal{T})$ has significant measure.} In view of \eqref{card-bound}, by taking $b$ to be large enough depending only on $q$, $C_\XX$, and $\tau$ (hence only on $q$, $C_\XX$, $m$, and $s_P$) and $\alpha>0$ to be sufficiently small depending on $q$, $C_\XX$, $b$, and $\tau$ (hence only on $q$, $C_\XX$, $b$, $m$, and $s_P$), we may arrange things so that for all $Q\in\Delta_+$, \begin{equation}\label{art-bound} L_Q\alpha + \frac{s_P-L_Q\alpha\log_b(1/\alpha)}{\log_b(M_Q-1)} \leq C_\XX b^q\alpha+\frac{s_P}{q-\log_b(2C_\XX)} \leq (1+\tau)\frac{s_P}{q}=:t,\end{equation} where $L_Q=\#\Outer(Q)$ and $M_Q=\#\Inner(Q)$. Since $k_{j+1}/n_{j+1}=a=t+\tau s_P/q>t$, from \eqref{gamma-bound},  \eqref{measure-bound-t}, and \eqref{art-bound}, we conclude that for all $j\geq 0$ and all $Q\in\mathcal{T}_j$, \begin{equation}\muse\left(\textstyle\bigcup K_Q(n_{j+1},k_{j+1})\right)\geq \left(1-\exp\left(-n_{j+1}\frac{\tau^2s_P^2}{2q^2}\right)\right)\muse(Q).\end{equation} Thus, by continuity from above, \begin{equation}\label{art-measure}
\muse(\leaves(\mathcal{T})) = \lim_{j\rightarrow \infty} \muse\left(\textstyle\bigcup\mathcal{T}_j\right)=\muse(Q_0)\prod_{i=1}^\infty \left(1-\exp\left(-in_0\frac{\tau^2s_P^2}{2q^2}\right)\right)>0,\end{equation} where the infinite product is positive, since $0< c:= \exp(-n_0\tau^2s_P^2/(2q^2))<1$. As the parameter $n_0\rightarrow\infty$, the number $c\rightarrow 0$ and the infinite product in \eqref{art-measure} tends to $1$. Therefore, by taking $n_0$ to be large enough depending only on $q$, $s_P$, $\tau$, and $\theta$ (hence only on $q$, $m$, $s_P$, and $\theta$), we obtain $\muse(\leaves(\mathcal{T}))\geq \theta\muse(Q_0)$.

\emph{We show that $\leaves(\mathcal{T})$ is contained in a Lipschitz image of $E\subset\RR^m$.} Write $\hat N:=C_\XX b^q$. Taking $b$ to be large enough depending only on $C_\XX$, $q$, $s_P$, and $\tau$ (hence only on $C_\XX$, $q$, $s_P$, and $m$) ensures that $a=(1+2\tau)s_P/q<\hat N/(\hat N+1)$. By Lemma \ref{l:KQ-bound}, for all $j\geq 0$ and $Q\in\mathcal{T}_j$, \begin{equation}\#\Child_\mathcal{T}(Q)=\#K_Q(n_{j+1},k_{j+1}) \leq b^{(j+1)n_0\left((1+2\tau)(s_P/q)\log_b\hat N+\log_b(2)\right)}\end{equation} Note that $\log_b(\hat N)=q+\log_b(C_\XX)$. Increasing the scaling factor $b$ as necessary (depending only on $q$, $C_\XX$, $m$, and $s_P$), we can arrange for  $(1+2\tau)(s_P/q)\log_b(C_\XX)+\log_b(2)\leq \tau s_P$. Then \begin{equation} N_j=\max_{Q\in\mathcal{T}_j}\#\Child_\mathcal{T}(Q) \leq b^{(j+1)n_0(1+3\tau)s_P},\quad D_j=\max_{Q\in\mathcal{T}_j}\diam Q\leq 2C_b\,b^{-\frac{1}{2}j(j+1)n_0}.\end{equation} Hence \begin{equation}\label{check-S}
\check S=\sum_{j=0}^\infty \left(\prod_{i=0}^j N_i^{1/m}\right)D_j\leq 2C_b\sum_{j=0}^\infty b^{\frac12 n_0\left[(j+1)(j+2)(1+3\tau)(s_P/m)-j(j+1)\right]}.\end{equation} Because $\eta:=1-(1+3\tau)s_P/m>0$, see \eqref{tau61}, the expression in square brackets $$[\cdots]\leq -\eta j^2+O(j)\leq -\frac{1}{2}\eta j^2\quad\text{for }j\gg_\eta 1.$$ (Here the big-$O$ notation means $O(j)\lesssim j$ for all $j\geq 1$.) Thus, we certainly know that $\check S<\infty$. We are almost ready to invoke Theorem \ref{t:pack}.

To proceed, note that for any $j\geq 0$ and any $Q\in\mathcal{T}_j$, we can bound \begin{equation}N_j\geq \#\Child_\mathcal{T}(Q)=\# K_Q(n_{j+1},k_{j+1})\geq \#\Child_{\Delta}^{k_{j+1}}(Q)\geq C_\XX^{-1}b^{(j+1)n_0(1+2\tau)s_P},\end{equation} where the final inequality is by \eqref{card-bound}. (The penultimate inequality holds, because there is a bijection between $\Child_{\Delta}^{k_{j+1}}(Q)$ and the set of all $R\in\Child_\Delta^{n_{j+1}}(Q)$ such that $\mathsf{ci}(R)=1$ and whose $n_{j+1}-k_{j+1}-1$ immediate ancestors $P$ also satisfy $\mathsf{ci}(P)=1$.) It follows that $C=\sum_{j=0}^\infty N_j^{-1/m} <\infty$, and thus, in addition to $\check S<\infty$, we have \begin{equation}S=\sum_{j=0}^\infty \left(\prod_{i=0}^j \lceil N_i^{1/m}\rceil\right)D_j\leq e^C\check S<\infty \end{equation} by Remark \ref{r:no-ceil}. Therefore, by Theorem \ref{t:pack}, there exists a compact set $E\subset\RR^m$ and a Lipschitz map $f:E\rightarrow\XX$ such that $f(E)\supset \leaves(\mathcal{T})$. This completes the proof of \eqref{m-rectifiable}.

\subsection{Doubling measures on \texorpdfstring{$\RR^d$}{Rd} vanish on porous sets}\label{ss:porous} To complete the proof of the theorem, all that remains is to verify \eqref{bi-lip-vanish}. The argument is now standard. Assume that $\XX=\RR^d$ and $1\leq m\leq d-1$. Let $\mu$ be any doubling measure on $\RR^d$ with full support. If $g:\RR^m\rightarrow\RR^d$ is a bi-Lipschitz embedding, then $\Sigma=g(\RR^m)$ is $m$-Ahlfors regular. In particular, $\dim_A \Sigma=m\leq d-1$. More generally, suppose that $\Sigma\subset\RR^d$ is any set with $\dim_A \Sigma<d$. By \cite[Theorem 5.2]{Luukk}, $\Sigma\subset\RR^d$ and $\dim_A \Sigma<d$ imply that $\Sigma$ is porous in the sense that there exists $0<\epsilon<1$ such that \begin{quotation}for all $x\in \Sigma$ and $r>0$, there exists a ball $B(y,\epsilon r)\subset B(x,r)$ such that $B(y,\epsilon r)\cap \Sigma=\emptyset$. \end{quotation} Let $x\in\Sigma$, let $r>0$, and let $B(y,\epsilon r)$ be the ball given by the porosity condition. Since $\mu$ is doubling and $y$ is in the support of $\mu$ (simply because $\spt\mu=\RR^d$), we have $\mu(B(x,r)\setminus\Sigma)\geq \mu(B(y,\epsilon r))\gtrsim \mu(B(y,2r))\geq \mu(B(x,r))$, where the implicit constant depends only on $\epsilon$ and the doubling constant for $\mu$. Thus, on the one hand, \begin{equation} \label{zero-density-at-x} \lim_{r\downarrow 0} \frac{\mu(B(x,r)\setminus\Sigma)}{\mu(B(x,r))}=0\end{equation} fails at every $x\in\Sigma$. On the other hand, by the Lebesgue-Besicovitch differentiation theorem for Radon measures (applied to the characteristic function $\chi_{\RR^d\setminus\Sigma}$), we know that \eqref{zero-density-at-x} holds at $\mu$-a.e.~$x\in\Sigma$ (see e.g.~\cite[Corollary 2.4]{Mattila}). Therefore, $\mu(\Sigma)=0$.


\section*{Acknowledgments} 
Research for this project began in 2020, during the early days of the COVID-19 pandemic. The authors thank their family members---Matt, Maya, Naomi, and Mikael---for patience with mathematical discussions over online meetings from home. We also thank Pablo Shmerkin and Boris Solomyak for explaining \eqref{limsup-infinity} and their help in locating references. Finally, the authors are grateful to the referee for their comments on an earlier draft of this manuscript.


\providecommand{\MR}{\relax\ifhmode\unskip\space\fi MR }
\providecommand{\MRhref}[2]{%
  \href{http://www.ams.org/mathscinet-getitem?mr=#1}{#2}
}
\providecommand{\href}[2]{#2}


\begin{dajauthors}
\begin{authorinfo}[badger]
  Matthew Badger\\
  University of Connecticut\\
  Storrs, Connecticut, USA\\
  matthew\imagedot{}badger\imageat{}uconn\imagedot{}edu \\
  \url{https://tangentmeasure.com/}
\end{authorinfo}
\begin{authorinfo}[schul]
  Raanan Schul\\
  Stony Brook University\\
  Stony Brook, New York, USA\\
  raanan\imagedot{}schul\imageat{}stonybrook\imagedot{}edu\\
  \url{https://www.math.stonybrook.edu/~schul/}
\end{authorinfo}
\end{dajauthors}

\end{document}